\documentclass[preprint,12pt]{elsarticle}

\usepackage{amssymb}
\usepackage{amsmath}

\usepackage{graphicx}
\usepackage{subcaption}	
\usepackage{multirow}
\usepackage{amsfonts}
\usepackage{amsthm}
\usepackage{mathrsfs}
\usepackage{xcolor}
\usepackage{textcomp}
\usepackage{manyfoot}
\usepackage{booktabs}
\usepackage{algorithm}
\usepackage{algorithmicx}
\usepackage{algpseudocode}
\usepackage{listings}
\usepackage{enumitem}

\newtheorem{theorem}{Theorem}
\newtheorem{proposition}[theorem]{Proposition}
\newtheorem{lemma}[theorem]{Lemma}

\theoremstyle{definition}
\newtheorem{definition}[theorem]{Definition}

\theoremstyle{remark}
\newtheorem{remark}[theorem]{Remark}

\theoremstyle{plain}
\newtheorem{corollary}[theorem]{Corollary}

\newcommand{\N}{\mathbb{N}}

\newcommand{\R}{\mathbb{R}}

\newcommand{\U}{\mathbb{U}}

\newcommand{\G}{\mathcal{G}}

\newcommand{\UU}{\mathcal{U}}

\newcommand{\ZZ}{\mathcal{Z}}
\newcommand{\A}{\mathcal{A}}
\newcommand{\AB}{\partial \mathcal{A}}
\newcommand{\AM}{[\partial \A]_-}
\newcommand{\AO}{[\partial \A]_0}

\newcommand{\AP}{A_{\mathcal{P}}}
\newcommand{\bP}{b_{\mathcal{P}}}

\journal{IFAC Journal of Systems and Control}

\begin{document}

\begin{frontmatter}



\title{Polytopic Inner Approximation of Admissible Sets for Linear Systems}

 \author[label1]{Jean L\'{e}vine}
 \author[label2]{Philipp Rumschinski}
 \author[label3]{Franz Ru{\ss}wurm}
 \author[label3]{Stefan Streif\corref{cor1}}
 
 \cortext[cor1]{Corresponding author. Email: stefan.streif@etit.tu-chemnitz.de}

 \affiliation[label1]{organization={Unit\'{e} Maths et Syst\`{e}mes, MINES Paris - PSL University},
             addressline={60 Bd Saint-Michel},
             city={Paris},
             postcode={75272},
             country={France}}
             
\affiliation[label2]{organization={Faculty of Engineering and Technology, Hochschule Furtwangen},
             addressline={Robert-Gerwig-Platz 1},
             city={Furtwangen},
             postcode={78120},
             country={Germany}}
    
\affiliation[label3]{organization={Automatic Control and System Dynamics, Chemnitz University of Technology},
             addressline={Str. der Nationen 62},
             city={Chemnitz},
             postcode={09111},
             country={Germany}}

%

\begin{abstract}
This paper presents a method for computing inner polytopic approximations of admissible sets for continuous-time linear control systems subject to multiple affine state constraints, with a particular concern on computational tractability for large dimensional problems.
In place of globally computing the admissible set and the part of its boundary called the barrier, we compute the so-called individual  admissible sets and the corresponding barriers for each single constraint.
We then use the exact sampling of linear systems and generate polytopes in half-space representation that provide an approximation of the individual admissible sets, to finally intersect them.
We provide a complexity analysis of the whole procedure to evaluate its efficiency.
The approach is illustrated by two examples —a triple integrator and a mass–spring–damper chain considered in 4, 6, 8, and 10 dimensions— with corresponding runtimes evaluated for both.

%
%
\end{abstract}



\begin{keyword}
linear system \sep affine state constraints \sep continuous-time system \sep admissible set \sep barrier theory \sep convex sets



\end{keyword}

\end{frontmatter}



\section{Introduction}\label{Sec:Introduction}
Admissible sets play a central role in constrained control, since they describe the set of initial conditions for which the state constraints can be satisfied for all future times under a suitable input. Hence, their computation are of interest in constrained stabilization, safety verification, and optimization-based control.
However, if the dimension of the system increases, the construction of admissible sets becomes computationally challenging.

In literature, several classes of methods have been developed to approximate such sets. Grid-based viability and reachability methods, including Hamilton-Jacobi methods, typically rely on discretization of the state space and the computation of an implicit representation of the set on this grid \cite{Mitchell_2005, Gillula_2014, Chen_2017}. While these methods are powerful and broadly applicable, their computational cost increases rapidly with the state dimension. Moreover, refinement usually requires recomputation on a finer grid, and the resulting grid-based representation is not always directly suitable for constrained control or optimization unless combined with a set description as proposed in \cite{Gillula_2014}.

Another important class of approaches is based on set propagation and reachability analysis. For instance, reachable sets may be propagated through time with help of support functions \cite{Le_Guernic_2009}, or through the construction of ellipsoidal inner approximations \cite{Kurzhanski_2000}. Such methods provide rigorous geometric approximations, but they require the repeated propagation of sets or geometric objects whose complexity depends strongly on the chosen set representation. Similarly, zonotope-based approaches propagate or scale structured set representations to obtain invariant or reachable-set approximations \cite{Mitchell_2019, Wetzlinger_2025}. Hereby, the geometry of zonotopic approximations is determined by the selected generators and the central symmetry property of zonotopes, which might limit their ability to approximate non-symmetric admissible set boundaries.

An alternative viewpoint is provided by barrier theory, a particular case of viability kernels, see \cite{Aubin_1991}.
One of the specificities of barrier theory is that it provides a characterization of the boundary of the admissible set based on trajectories satisfying necessary conditions similar to Pontryagin's maximum principle \cite{Levine_2013}. Accordingly, the method focuses on the relevant part of the state space, namely the admissible set boundary.

In order to reduce the curse of dimensionality when the number of constraints is high, we propose in this paper to decompose the global computation of the admissible set by an individual computation of admissible sets associated to each individual constraint.
This will be done in the framework of linear systems subject to affine state constraints to concur with the aim of decreasing the complexity.


This paper combines the geometric boundary characterization provided by barrier theory with a classical use of exact sampling available in this context.
Instead of integrating individual barrier trajectories, the proposed sampling algorithm propagates the set of boundary trajectories together by a suitable parameterization.
The resulting samples are converted into polytopic inner approximations of the admissible set in half-space representation, thanks to the convexity property of the individual and global admissible set.
%
Finally, runtime complexity analyses are derived for both the structured sampling and polytope reconstruction algorithms.

The remainder of the paper is organized as follows. Section~\ref{Sec:Barrier_Theory_for_Linear_Control_Systems} formulates the admissible set problem, specializes the barrier-theoretic characterization to constrained linear systems and establishes structural properties of the resulting admissible sets, and we mention a possible extension to Hammerstein systems.
Section~\ref{Sec:Structured_Barrier_Sampling} develops the proposed structured sampling algorithm and derives its computational complexity.
Section~\ref{Sec:Construction_Inner_Approximation} presents the construction of polytopic inner approximations in half-space representation from the obtained samples together with the corresponding complexity analysis. 
Numerical results for high-dimensional constrained systems are presented in Section~\ref{Sec:Examples}, illustrating the proposed algorithms and the computational benefit of exploiting the barrier-induced boundary structure.
Finally, Section~\ref{Sec:Conclusion} concludes the paper.

\section{Barrier Theory for Linear Control Systems}\label{Sec:Barrier_Theory_for_Linear_Control_Systems}

We consider a linear control system of the form
\begin{align}\label{eq:Linear_System}
\begin{split}
\dot{x}(t) &= A x(t) + B u(t) \\
x(t_0) &= x^0
\end{split}
\end{align}
with the state $x(t) \in \R^n$, initial value $x^0 \in \R^n$ and matrices $A \in \R^{n \times n}$ and $B \in \R^{n \times m}$.
We assume the set $\U \subset \R^m$ is nonempty, convex and compact.
Let $\UU$ be the set of Lebesgue-measurable functions $u:[t_0,\infty) \rightarrow \U$.
We denote by $x^{(x_0,u)}$, or $x^{u}$ if unambiguous, the absolutely continuous maximal integral curve that satisfies \eqref{eq:Linear_System} with initial value $x_0 \in \R^n$ and generated by the control law $u \in \UU$.

In addition, we consider affine state constraints of the form
\begin{align}\label{eq:Linear_Constraints}
H x(t) + h \preceq 0 \quad \forall \, t \in [t_0,\infty)
\end{align}
with $H \in \R^{p \times n}$ and $h \in \R^p$ for some $p \in \N$, where the sign $a \preceq b$ means $a_i \leq b_i$ for all $i = 1,2\ldots,p$.

Moreover, we define the $i$-th individual state constraint as
\begin{align}\label{eq:Single_Linear_Constraint}
H_i x(t) + h_i \leq 0  \quad \forall \, t \in [t_0,\infty),
\end{align}
where $H_i \in \R^{1 \times n}$ the $i$-th row of the matrix $H$ and by $h_i$ the $i$-th component of the vector $h$.
We indeed assume that for all $i \in \{1,2,\ldots,p\}$, $H_i$ is nonzero, i.e., $H_i \neq 0 \in \R^n$.

We define the constraint set $\G$ as $\G \triangleq \{x \in \R^n \, \vert \, Hx + h \preceq 0 \}$.
The set $\G$ can be split into its interior $\G_- \triangleq \{x \in \R^n \, \vert \, Hx + h \prec 0 \}$ and its boundary $\G_0 \triangleq \{x \in \R^n \, \vert \, \exists \, i \in \{1,2,\ldots,p \}: H_i x + h_i = 0 \}$.
We assume $\G$ to have a nonempty interior and $\G_0$ to be nonempty.
It holds that $\G = \G_0 \: \cup \: \G_-$.

We denote by $\G_i \triangleq \{x \in \R^n \, \vert \, H_i x + h_i \leq 0 \}$ the \emph{individual constraint set} associated with the $i$-th individual state constraint, with analogous definitions of its interior $[\G_i]_-$ and its boundary $[\G_i]_0$.

\begin{remark}
	The assumptions introduced in \cite{Levine_2013} in the nonlinear case are automatically satisfied for linear systems.
	Their straightforward verification is left to the reader.
\end{remark}

We recall from \cite{Levine_2013} the definition of the admissible set
\begin{align*}
\A \triangleq \{ \bar{x} \in \G \, \vert \, \exists  u \in \UU : x^{(\bar{x},u)}(t) \in \G \; \; \forall t \geq 0 \},
\end{align*}
which is proved to be closed and containing its boundary $\AB$.
This boundary can be split into the usable part $\AO = \partial \A \cap \G_0$ and the barrier $\AM = \partial \A \cap \G_-$, in accordance with \cite{Isaacs_1965}.

We now consider the so-called \emph{individual admissible set} associated with each individual constraint $i\in\{1,2,\ldots,p\}$
\begin{align*}
\A_i \triangleq \{ \bar{x} \in \G_i \, \vert \, \exists  u \in \UU : x^{(\bar{x},u)}(t) \in \mathcal{G}_i \; \; \forall t \geq 0 \}.
\end{align*}
The corresponding boundary components $[\partial \A_i]_-$ and $[\partial \A_i]_0$ are defined analogously as above.

\subsection{Barrier-Theoretic Characterization of the Admissible Sets and Individual Admissible Sets}\label{Sec:Barrier_Theoretic_Characterization_of_the_Admissible_Set}
We recall the barrier-theoretic characterization of admissible sets introduced in \cite{Levine_2013} for the linear system \eqref{eq:Linear_System} subject to the affine state constraints \eqref{eq:Linear_Constraints}, with the assumptions:

\newpage

\begin{enumerate}[before=\textit{Assumptions \vspace{-0.5em}}, label=(A\arabic*)]
    \item There exists an $\bar{x} \in \A$ and a $u\in \UU$ such that
    \begin{align*}
        \sup_{t \in [0,\infty)} \max_{i = 1,2,\ldots,p} \left( H_i x^{(\bar{x},u)}(t) + h_i \right) = 0.
    \end{align*}
    \item For every $i \in \{1,2,\ldots,p\}$,
    \begin{align*}
        \dim \left( \ker \left\lbrace \begin{pmatrix}
			H_i \\ H_i A
		\end{pmatrix} \right\rbrace \right) = n-2.
    \end{align*}
\end{enumerate}

\begin{remark}
Assumption (A1) guarantees that the admissible set is nonempty and intersects the constraint boundary $\mathcal{G}_0$, while (A2) ensures that the tangentiality condition is nondegenerated.
\end{remark}

Let
\begin{align*}
        \mathbb{I}(x) \triangleq \left\lbrace i \in \{1,2,\ldots,p \} \; \vert \; H_i x + h_i = 0 \right\rbrace . 
\end{align*}
be the set of active indices at a point $x \in \R^n$.

By Theorem~7.1 of \cite{Levine_2013}, a trajectory $x^{u}$ belongs to the barrier $[\partial \mathcal{A}]_-$ if it evolves along $\partial \mathcal{A}$ and satisfies the adjoint equation \eqref{eq:General_Adjoint_Equation} and the Hamiltonian condition \eqref{eq:General_Cond_Hamiltonian}, with terminal conditions determined by the ultimate tangentiality condition \eqref{eq:General_Condition_ultimate_tangentiality}:
\begin{align}\label{eq:General_Adjoint_Equation}
\dot{\lambda}(t) = -A^\top \lambda(t),
\quad
\lambda(0) = H_{i^*}^\top,
\end{align}
where $i^* \in \mathbb{I}(\zeta)$ such that

\begin{align}\label{eq:General_Condition_ultimate_tangentiality}
    \min_{u\in \mathbb{U}} \max_{i \in \mathbb{I}(\zeta)} H_i (A \zeta + B u) =  H_{i^*} (A \zeta + B \bar{u}) = 0
\end{align}

and 

\begin{align}
    \begin{split}\label{eq:General_Cond_Hamiltonian}
        \min_{u \in \mathbb{U}} \, \left\lbrace (  \lambda^{u}(t))^\top \left(A x^{u}(t) + Bu \right) \right\rbrace = 0.
    \end{split}
\end{align}

%

\subsection{The Individual Admissible Sets}\label{Sec:Individual_Admissible_Sets}
For the purpose of computing equations \eqref{eq:General_Adjoint_Equation}, \eqref{eq:General_Condition_ultimate_tangentiality} and \eqref{eq:General_Cond_Hamiltonian} in an efficient way, we introduce the following characterization of the individual admissible sets and their boundary.

The barrier-theoretic characterization from the previous subsection can be applied directly to systems subject to multiple state constraints by considering all active constraints simultaneously. 
However, in this setting, barrier trajectories associated with different constraints may intersect, leading to so-called stopping points \cite{Esterhuizen_2014}, where the backward integration must be terminated. 
Determining such stopping points is, in general, nontrivial and depends strongly on the underlying system dynamics.

To avoid this difficulty, we instead consider the admissible sets associated with the individual state constraints separately and construct the admissible set corresponding to multiple constraints through set intersections.

As a straightforward application of Theorem~7.1 of \cite{Levine_2013}, the characterization for an individual state constraint $i \in \{1,2,\ldots p\}$ is that the barrier $[\partial \mathcal{A}_i]_-$ consists of trajectories $x^{u}$ that evolve along $\partial \mathcal{A}_i$ satisfying \eqref{eq:Adjoint_Equation} with $\lambda_i \neq 0$ and \eqref{eq:Cond_Hamiltonian} with final conditions given by the ultimate tangentiality condition \eqref{eq:Condition_ultimate_tangentiality}:

\begin{align}\label{eq:Adjoint_Equation}
\dot{\lambda}_i(t) = -A^\top \lambda_i(t),
\quad
\lambda_i(0) = H_{i}^\top,
\end{align}

\begin{align}\label{eq:Condition_ultimate_tangentiality}
    \min_{u\in \mathbb{U}} H_i (A \zeta_i + B u) = 0
\end{align}

\begin{align}
    \begin{split}\label{eq:Cond_Hamiltonian}
        \min_{u \in \mathbb{U}} \, \left\lbrace (  \lambda_i(t))^\top \left(A x^{u}(t) + Bu \right) \right\rbrace = 0.
    \end{split}
\end{align}

The barrier control law $\bar{u}$ is obtained from condition \eqref{eq:Cond_Hamiltonian}.
The explicit solution of the adjoint system \eqref{eq:Adjoint_Equation} reads
\begin{align*}
\lambda_i(t) = e^{-A^\top t} H_i^\top,
\end{align*}
where, without loss of generality, we have set $\bar t=0$.
Substituting this explicit solution into \eqref{eq:Cond_Hamiltonian}, we obtain
\begin{align*}
0 &= \min_{u \in \U} \left\lbrace \lambda_i(t)^\top A x^{\bar{u}}(t) + \lambda_i(t)^\top B u \right\rbrace \\
&= \lambda_i(t)^\top A x^{\bar{u}}(t) +  \min_{u \in \U} \left\lbrace \lambda_i(t)^\top B u \right\rbrace \\
&= \left( A^\top e^{-A^\top t} H_i^\top \right)^\top x^{\bar{u}}(t) +  \min_{u \in \U} \left\lbrace \left( B^\top e^{-A^\top t} H_i^\top \right)^\top u \right\rbrace.
\end{align*}
Thus, the barrier control law satisfies
\begin{align}\label{eq:Barrier_Control}
\bar{u}(t) = \underset{u \in \U}{\textnormal{argmin}} \, \sigma_i(t)^\top u
\end{align}
with the switching function
\begin{align}\label{eq:Switching_Function}
\sigma_i(t) \triangleq B^\top e^{-A^\top t} H_i^\top .
\end{align}
Remark the switching function is independent of $x$.

We use the the final condition \eqref{eq:Condition_ultimate_tangentiality} to obtain the set of ultimate tangentiality points
\begin{align}\label{eq:Def_UT-set}
\ZZ_i \triangleq \left\{ \zeta \in \R^n \, \vert \, H_i \zeta = -h_i, \, H_i A \zeta = -\min_{u\in \U} \, H_i B u \right\}.
\end{align}


Consequently, trajectories on $[\partial \A_i]_-$ are obtained by integrating the system \eqref{eq:Linear_System} and the adjoint system \eqref{eq:Adjoint_Equation} backward in time from points $\zeta_i \in \mathcal{Z}_i$ under the barrier control law $\bar{u}$.

Let us denote by $\mathcal{B}_i$ the set of individual barrier trajectories, i.e., trajectories $x^{\bar{u}}(t)$ satisfying \eqref{eq:Linear_System} and \eqref{eq:Adjoint_Equation}-\eqref{eq:Def_UT-set} for $i = 1,2,\ldots,p$.
Then, we obtain the following lemma.

\begin{lemma}\label{Lem:Representation_Zi}
Let $\zeta_i$ be a particular solution of \eqref{eq:Def_UT-set}.
Then, $\ZZ_i = \{ \zeta_i \} + \mathcal{L}_i$, where we have denoted \begin{align*}
\mathcal{L}_i \triangleq \ker \left\lbrace \begin{pmatrix}
H_i \\ H_i A
\end{pmatrix} \right\rbrace,
\end{align*}
which is $(n-2)$-dimensional by Assumption (A2).
Then,
\begin{align*}
\mathcal{B}_i = \ZZ_i(t) = \{\zeta_i(t)\} + \mathcal{L}_i(t)
\end{align*}
with $\dot{\zeta}_i(t) = A \zeta_i(t) + B \bar{u}(t)$ and $ \mathcal{L}_i(t) = e^{At} \mathcal{L}_i$.
\end{lemma}

\begin{proof}
Take $\alpha \in \mathcal{L}_i$ and let $x(0) = \zeta_i(0) + \alpha$.
Then, we have $x(0) \in \ZZ_i$ and
\begin{align*}
\dot{x}(t) &= \dot{\zeta}_i(t) + \frac{d}{dt} e^{At} \alpha \\
&= \dot{\zeta}_i(t) + A e^{At} \alpha \\
&= A \zeta_i(t) + B \bar{u}(t) + A e^{At} \alpha \\
&= A \left(\zeta_i(t) + e^{At} \alpha \right) + B \bar{u}(t) \\
&= A x(t)+ B \bar{u}(t).
\end{align*}
Hence, $x(t)$ satisfies \eqref{eq:Linear_System} with the initial condition $x(0)=\zeta_i(0)+\alpha$.
By uniqueness of solutions,
\begin{align*}
x(t)=\zeta_i(t)+e^{At}\alpha.
\end{align*}
Therefore, every trajectory generated from an initial point in $\ZZ_i$ belongs to
\begin{align*}
\{\zeta_i(t)\} + e^{At} \mathcal{L}_i.
\end{align*}

Conversely, for every $\alpha \in \mathcal{L}_i$, the trajectory
\begin{align*}
x(t)=\zeta_i(t)+e^{At}\alpha
\end{align*}
satisfies \eqref{eq:Linear_System} under the barrier control law $\bar{u}$ and has initial condition $x(0)=\zeta_i(0)+\alpha \in \ZZ_i$.
Thus, every point in ${\zeta_i(t)}+e^{At}\mathcal{L}_i$ is generated by a barrier trajectory originating from $\ZZ_i$.
The lemma is proven.
\end{proof}

\begin{remark}\label{Rem:UT_points_n-2_dimensional}
The set $\mathcal{Z}_i$ is, in general, an $(n-2)$-dimensional manifold, since it is defined by two independent scalar conditions.
Backward integration from $\mathcal{Z}_i$ under the barrier control law adds the integration time as one degree of freedom.
Consequently, the barrier $[\partial\mathcal A_i]_-$ is locally an $(n-1)$-dimensional manifold.
\end{remark}

\subsection{Convexity of Individual Admissible Sets and the Global One}\label{Sec:Structural_Properties_of_the_Admissible_Set}
The following result shows that the admissible sets associated with the individual state constraints and the global one are convex.
These results will play an important role in Section \ref{Sec:Construction_Inner_Approximation}, in order to generate the whole admissible set from its boundary, using the fact that convex combinations of boundary samples remain inside the corresponding admissible set, and allowing to generate polytopic approximations in half-space form based on the barrier points. 

\begin{proposition}\label{Prop:Ai_convex}
For all $i \in \{1,2,\ldots,p \}$, the individual admissible set $\A_i$ of system \eqref{eq:Linear_System} subject to an affine state constraint of the form \eqref{eq:Single_Linear_Constraint} is convex.
\end{proposition}

\begin{proof}
Let $x^0_1,x^0_2 \in \A_i$ with corresponding admissible control inputs $u_1, u_2 \in \UU$ such that $x^{(x^0_1,u_1)}(t), x^{(x^0_2,u_2)}(t) \in \A_i$ for all $t \geq 0$, and let $\rho \in [0,1]$.

Since the system is linear,
\begin{align*}
& \quad \; x^{(\rho x^0_1 + (1-\rho) x^0_2, \rho u_1 + (1-\rho) u_2)}(t) \\ 
&= e^{At}(\rho x^0_1 + (1-\rho) x^0_2) + \int_0^t \, e^{A(t-s)}B(\rho u_1(s) + (1-\rho) u_2(s)) \, ds \\
&= \rho \left( e^{At} x^0_1 + \int_0^t \, e^{A(t-s)}Bu_1(s) \, ds \right) + (1-\rho) \left( e^{At}x^0_2 + \int_0^t \, e^{A(t-s)}Bu_2(s) \, ds \right) \\
&=\rho x^{(x^0_1,u_1)}(t) + (1-\rho)x^{(x^0_2,u_2)}(t).
\end{align*}
Moreover, the convexity of $\U$ implies that
\begin{align*}
\rho u_1 + (1-\rho)u_2 \in \UU.
\end{align*}
Thus, we obtain
\begin{align*}
& \quad \; H_i \left( \rho x^{(x^0_1,u_1)}(t) + (1-\rho)x^{(x^0_2,u_2)}(t) \right) + h_i \\
&= \rho \left( H_i x^{(x^0_1,u_1)}(t) + h_i \right) + (1-\rho) \left( H_i x^{(x^0_2,u_2)}(t) + h_i \right) \leq 0
\end{align*}
for all $t \geq 0$.
Hence,
\begin{align*}
    \rho x_1^0+(1-\rho)x_2^0 \in \A_i,
\end{align*}
which proofs the claim.
\end{proof}

We immediately obtain the following result.

\begin{corollary}\label{Cor_A_convex}
The admissible set $\A$ of system \eqref{eq:Linear_System} is convex.
\end{corollary}

\begin{proof}
By Proposition \ref{Prop:Ai_convex}, the individual admissible sets $\A_i$ are convex for $i = 1,2,\ldots,p$.
Since the intersection of convex sets is again convex and
\begin{align*}
\A = \bigcap_{i=1}^{p} \A_i,
\end{align*}
the result follows immediately.
\end{proof}


\begin{proposition}\label{Prop:Normal_Test}
For every point $x_b \in \partial \mathcal{A}_i$, let $n(x_b)$ denote an outward normal vector of $\mathcal{A}_i$ at $x_b$.

Then, $x \in \mathcal{A}_i$ if and only if
\begin{align*}
n(x_b)^\top (x_b-x) \geq 0
\end{align*}
for every boundary point $x_b \in \partial \mathcal{A}_i$.
\end{proposition}

\begin{proof}
The proof follows from \cite[Corollary 11.5.2]{Rockafellar_1970}.
\end{proof}


\begin{remark}
Consider the following extension to Hammerstein systems
\begin{align}\label{eq:Hammerstein_System}
\begin{split}
	\dot{x}(t) &= Ax(t) + B\varphi(u(t)) \\
	x(t_0) &= x^0,
\end{split}
\end{align}
where $\varphi : \mathbb R^m \to \mathbb R^m$ denotes a static nonlinear input transformation. 

Introducing the virtual input $v(t) = \varphi(u(t))$ yields the transformed system
\begin{align}\label{eq:Hammerstein_System_transformed}
\begin{split}
	\dot{x}(t) &= Ax(t) + Bv(t)\\
	x(t_0) &= x^0.
\end{split}
\end{align}
Assuming $\phi(\mathbb{U}) = \mathbb{V}$ is convex and compact, the previous assumptions are satisfied.
Hence, since the constraints are not modified, Proposition \ref{Prop:Ai_convex}, Corollary \ref{Cor_A_convex} and Proposition \ref{Prop:Normal_Test} can be extended to system \eqref{eq:Hammerstein_System}.
\end{remark}

\section{Exact Individual Barrier Sampling}\label{Sec:Structured_Barrier_Sampling}



Consider the linear control system \eqref{eq:Linear_System}.
For a given initial condition $x(t_0) = x^0$ and an admissible control input $u(\cdot) \in \mathcal{U}$, the solution of this system can be expressed explicitly.
For all $t \ge t_0$, the unique solution is given by
\begin{align}\label{eq:explicit_solution_linear}
x^{(x^0,u)}(t) = e^{A(t - t_0)} x^0 + \int_{t_0}^{t} e^{A(t - \tau)} B u(\tau) \, \mathrm{d} \tau.
\end{align}
This representation holds for any measurable and bounded control input $u(\cdot)$.

Suppose now that the control input is piecewise constant on time intervals of the length $\delta > 0$, i.e.,
\begin{align*}
u(t) = u(k \delta) = \textnormal{const.} \quad \forall \, t \in [k \delta, (k+1) \delta), \, k \in \mathbb{N} .
\end{align*}
Then, the continuous-time system \eqref{eq:Linear_System} admits an exact discrete-time representation of the form
\begin{align}\label{eq:Discrete_Linear_System}
\begin{split}
x(k+1) &= A_d x(k) + B_d u(k) \\
x(0) &= x^0,
\end{split}
\end{align}
where the discrete-time system matrices are given by
\begin{align*}
A_d = e^{A \delta}, \qquad B_d = \int_0^\delta e^{A \tau} \, \mathrm{d} \tau B,
\end{align*}
see, e.g., \cite{Lunze2_2020}[Chapter 11.1.6].
The matrices $A_d$ and $B_d$ can be computed efficiently using the zero-order hold representation
\begin{align*}
    \begin{pmatrix}
        A_d & B_d \\
        0 & I
    \end{pmatrix} = \exp \left( \begin{pmatrix}
        A & B \\
        0 & 0
    \end{pmatrix} \delta \right).
\end{align*}

\subsection{Algorithm for Exact Individual Barrier Sampling}
The characterization in Section \ref{Sec:Individual_Admissible_Sets} provides an exact representation of the individual barriers in continuous time.
In particular, by Lemma \ref{Lem:Representation_Zi}, the set of ultimate tangentiality points can be written as
\begin{align*}
\mathcal{Z}_i = \{ \zeta_i \} + \mathcal{L}_i,
\end{align*}
where $\zeta_i \in \mathcal{Z}_i$ and 
\begin{align*}
\mathcal{L}_i = \ker \left\lbrace \begin{pmatrix}
H_i \\ H_i A
\end{pmatrix} \right\rbrace.
\end{align*}
Since the switching function \eqref{eq:Switching_Function} is independent of the state $x$, all trajectories originating from $\mathcal{Z}_i$ share the same switching structure of $\bar{u}$.
If the discretization step $\delta$ is consistent with the switching times, the discrete-time system \eqref{eq:Discrete_Linear_System} therefore propagates these trajectories exactly over intervals with constant control.

At time step $k$, let
\begin{align*}
\mathcal{Z}_i(k) = \zeta_i(k) + \mathcal{L}_i(k),
\end{align*}
where $\zeta_i(k) \in \mathcal{Z}_i(k)$. Since the barrier trajectories are propagated backward from the constraint boundary, the discrete-time dynamics are inverted to obtain
\begin{align*}
x(k-1) = A_d^{-1}\left(x(k)-B_d u(k-1)\right).
\end{align*}
Consequently,
\begin{align*}
\zeta_i(k-1)
&= A_d^{-1}\zeta_i(k)-A_d^{-1}B_d u(k-1), \quad \mathcal{L}_i(k-1)= A_d^{-1}\mathcal{L}_i(k).
\end{align*}
Since $A_d=e^{A\delta}$ is nonsingular, the dimension of $\mathcal{L}_i(k)$ is preserved during the propagation.

For sampling, we represent $\mathcal{L}_i(k)$ by an orthonormal basis
\begin{align*}
\mathcal{L}_i(k) = \operatorname{span}(N_i(k)), \; N_i(k)\in \mathbb{R}^{n\times(n-2)}, \; N_i(k)^\top N_i(k)=I_{n-2}.
\end{align*}
Thus, every point $\zeta \in \mathcal{Z}_i(k)$ admits the representation
\begin{align}\label{eq:Sampling_Representation}
\zeta = \zeta_i(k)+N_i(k)\alpha, \quad \alpha\in\mathbb{R}^{n-2}.
\end{align}
The propagated direction space is obtained from
\begin{align*}
\mathcal{L}_i(k-1)
=\operatorname{span}\left(A_d^{-1}N_i(k)\right).
\end{align*}
Since $A_d^{-1}N_i(k)$ is generally not orthonormal, we set
\begin{align*}
N_i(k-1)
=\operatorname{orth}\left(A_d^{-1}N_i(k)\right),
\end{align*}
where $\operatorname{orth}(\cdot)$ denotes an orthonormalization procedure.

Hence, $\zeta_i(k)$ and $N_i(k)$ completely characterize $\mathcal{Z}_i(k)$, while the sampling resolution can be selected through $\alpha$ independently at each time step.

If $\mathcal{A}_i$ is unbounded, sampling is restricted to a region of interest $\mathcal{X}_{\mathrm{ROI}}\subset\mathbb{R}^n$.
To enlarge the resulting inner approximation, admissible corner points of $\mathcal{X}_{\mathrm{ROI}}$ and admissible intersections of $[\mathcal{G}_i]_0$ with $\mathcal{X}_{\mathrm{ROI}}$ are added to the set of samples.
Their admissibility is determined using Proposition~\ref{Prop:Normal_Test} and the outward normals obtained from the adjoint system \eqref{eq:Adjoint_Equation}.

All points in $\mathcal{Z}_i(k)$ share the same adjoint $\lambda_i(k)$, since the adjoint is independent of the state.
Its discrete-time propagation is given by
\begin{align*}
\lambda_i(k-1)=A_d^\top\lambda_i(k).
\end{align*}

The resulting sampling procedure is summarized in Algorithm~\ref{alg:Structured_Barrier_Sampling}.

\begin{algorithm}[htbp]
	\caption{Structured Sampling of Admissible Set's Boundary}\label{alg:Structured_Barrier_Sampling}
	\noindent \textbf{Input}: Continuous-time matrices $A \in \mathbb{R}^{n \times n}, B \in \mathbb{R}^{n \times m}$, discrete-time matrices $A_d \in \mathbb{R}^{n\times n}$, $B_d \in \mathbb{R}^{n\times m}$, constraint parameters $H_i \in \R^{1 \times n} ,h_i \in \R$, convex and compact set of control values $\mathbb{U} \subset \mathbb{R}^m$, region of interest $\mathcal{X}_{\mathrm{ROI}} \subset \R^n$, set $C_i$ of admissible corner points of $\mathcal{X}_{\mathrm{ROI}}$ and its intersection points with state constraint $g_i$, number of time steps $K \in \mathbb{N}$, finite parameter grid $\mathbb{G} \subset \mathbb{R}^{(n-2)}$. \\
	\textbf{Output:} Sample Set $\mathcal{S}_i$
	\begin{enumerate}
	\item Initialize $\mathcal{S}_i \gets \emptyset$, $\mathcal{B}_i \gets \emptyset$ and set $M \gets \begin{pmatrix} H_i \\ H_i A \end{pmatrix}$, $\lambda \gets H_i^\top$.
	\item Compute barrier control at ultimate tangentiality set $u^\ast = \mathrm{argmin}_{u \in \mathbb{U}} H_i B u$ and set $m \gets \begin{pmatrix} -h_i \\ -H_i B u^\ast \end{pmatrix}$.
	\item Compute one particular solution $\zeta_{\mathrm{ref}}$ of $M \zeta = m$. Compute orthonormal basis $N$ of $\ker(M)$.
	\item \textbf{For} each $\alpha \in \mathbb{G}$:
	\begin{enumerate}
		\item[(i)] $\zeta \gets \zeta_{\mathrm{ref}} + N \alpha$
		\item[(ii)] \textbf{If} $\zeta \in \mathcal{X}_{\mathrm{ROI}}$ and $H_i \zeta + h_i \leq 0$: $\mathcal{S}_i \gets \mathcal{S}_i \cup \{\zeta\}$, $\mathcal{B}_i \gets \mathcal{B}_i \cup \{(\zeta, \lambda)\}$.
	\end{enumerate}
	\item\textbf{For} $k = 1,2,\ldots,K$:
	\begin{enumerate}
		\item Set $\lambda \gets A_d^\top \lambda$.
		\item Compute barrier control $u^\ast = \mathrm{argmin}_{u \in \mathbb{U}} \lambda^\top B u$.
		\item Set $\zeta_{\mathrm{ref}} \gets A_d^{-1} \left( \zeta_{\mathrm{ref}} - B_d u^\ast \right)$. \\
        Compute orthonormal basis $N \gets \mathrm{orth}(A_d^{-1} N)$.
		\item \textbf{For} each $\alpha \in \mathbb{G}$:
		\begin{enumerate}
			\item[(i)] $\zeta \gets \zeta_{\mathrm{ref}} + N \alpha$
			\item[(ii)] \textbf{If} $\zeta \in \mathcal{X}_{\mathrm{ROI}}$ and $H_i \zeta + h_i \leq 0$: $\mathcal{S}_i \gets \mathcal{S}_i \cup \{\zeta\}$, $\mathcal{B}_i \gets \mathcal{B}_i \cup \{(\zeta,\lambda)\}$. 
		\end{enumerate}
	\end{enumerate}
    \item \textbf{For} every candidate point $x_c \in C_i$:\\
    \textbf{If} $\lambda_\zeta^\top(\zeta - x_c) \geq 0$ for all $(\zeta,\lambda_\zeta) \in \mathcal{B}_i$: $\mathcal{S}_i \gets \mathcal{S}_i \cup \{x_c\}$.
	\item Return $\mathcal{S}_i$.
	\end{enumerate}
\end{algorithm}

\subsection{Runtime Complexity of Sampling Algorithm}\label{Sec:Runtime_Sampling}
Let $n$ denote the state dimension, $m$ the input dimension, $K$ the number of time steps, and $G = | \mathbb{G} |$ the number of grid samples.
Throughout the analysis, we assume classical dense linear algebra operations, see, e.g., \cite[Chapter 1]{Golub_2013}.
In particular, dot products are $\mathcal{O}(n)$, dense matrix–vector multiplications scale as $\mathcal{O}(n^2)$ and matrix-matrix multiplications as  $\mathcal{O}(n^3)$.
The QR factorization of an $n \times (n-2)$ matrix requires $\mathcal{O}(n^3)$ floating point operations, see \cite[Chapter 5.2]{Golub_2013}.
Finally, assuming a standard dense interior-point method, solving a linear program in $m$ variables requires $\mathcal{O}(m^3)$ operations, see \cite{Vaidya_1987}.

During initialization, the construction of the matrix $M$ requires one dense multiplication and scales as $\mathcal{O}(n^2)$.
Solving the linear system $Mz = m$ and computing an orthonormal basis of $\ker(M)$ of $(n-2)$ vectors in $\R^n$ require matrix factorizations of size $n$, resulting in complexity $\mathcal{O}(n^3)$.
The barrier control is obtained by solving a linear optimization problem over the convex and compact input set $\mathbb{U} \subset \R^m$.
In the worst case, this corresponds to a linear program in $m$ variables with complexity $\mathcal{O}(m^3)$.
For structured representations of $\mathbb{U}$, such as box constraints or polytopes, the cost can be reduced to $\mathcal{O}(m)$.
Hence, the overall initialization has complexity $\mathcal{O}(n^3 + m^3)$.

For each parameter value $\alpha \in \mathbb{G}$ in the initial sampling step, the computation of $\zeta = \zeta_{\textnormal{ref}} + N \alpha$ requires a dense matrix–vector multiplication and therefore scales as $\mathcal{O}(n^2)$.
The subsequent membership test in the region of interest involves evaluating linear inequalities and is at most $\mathcal{O}(n^2)$. Repeating this for all $G$ samples yields a complexity $\mathcal{O}(G n^2)$.

During time propagation, each of the $K$ steps involves in (a) a dense multiplication $A_d^\top \lambda$ with complexity $\mathcal{O}(n^2)$, in (b) one barrier control computation with worst-case complexity $\mathcal{O}(m^3)$, and in (c) a propagation and re-orthonormalization of an $n \times (n-2)$ basis matrix, which scales as $\mathcal{O}(n^3)$.
The sampling step discussed during the initialization is then repeated in (d) with complexity $\mathcal{O}(Gn^2)$.
Hence, each time step has complexity $\mathcal{O}(n^3 + m^3 + G n^2)$.

Finally, the candidate points contained in $C_i$ are tested for admissibility using Proposition~\ref{Prop:Normal_Test}.
Let $C = | C_i |$ denote the number of candidate points and $S = | \mathcal{B}_i |$ the number of sampled boundary points.
For each $x_c \in C_i$, the inner product $\lambda_z^\top (z - x_c)$ is evaluated for every sampled boundary point.
Each evaluation requires $\mathcal{O}(n)$ operations, resulting in a total complexity of $\mathcal{O}(CSn)$.
Since only samples satisfying the state constraints are included in $\mathcal{B}_i$, we have $S \leq (K+1)G$ and therefore $\mathcal{O}(CSn) = \mathcal{O}(CKGn)$.

Collecting the costs of initialization, time propagation, and candidate-point classification gives the following result.

\begin{proposition}\label{Prop:Complexity_Sampling}
Under the assumptions stated in Section~\ref{Sec:Runtime_Sampling}, Algorithm~\ref{alg:Structured_Barrier_Sampling} has runtime complexity
\begin{align*}
\mathcal{O}\!\left( n^3 + m^3 + Gn^2 + K(n^3 + m^3 + Gn^2) + CKGn \right),
\end{align*}
where $n$ denotes the state dimension, $m$ the input dimension, $K$ the number of time steps, $G$ the number of grid samples, and $C$ the number of candidate points tested for admissibility.
In particular, the algorithm scales cubically in the state and input dimensions, and linearly in both the time horizon and the number of grid samples.
Furthermore, for $C \ll G$ and $G \gg n,m$, the complexity simplifies to
\begin{align*}
\mathcal{O}(KGn^2).
\end{align*}
\end{proposition}

\begin{proof}
The result follows by summing the computational costs of the individual steps derived above.
\end{proof}

\begin{remark}
For comparison, we briefly discuss the computational complexity of established set-based methods for admissible-set computation.

Hamilton--Jacobi-based methods \cite{Mitchell_2005,Chen_2017,Bansal_2017}
discretize the state space and therefore suffer from the curse of
dimensionality, with computational effort growing exponentially in the state dimension.

Sampling-based methods, such as \cite{Gillula_2014}, avoid a full
state-space discretization but require the solution of a feasibility problem for each sampled direction and time interval.
For a fixed number of time intervals, the complexity reported in \cite{Gillula_2014} is
\begin{align*}
	\mathcal O \left(N \log(d)\Phi(n)\right),
\end{align*}
where $N$ denotes the number of sampled directions, $d$ the diameter of the admissible set, and $\Phi(n)$ the complexity of the underlying feasibility program.

Set-propagation approaches based on geometric set representations, e.g., zonotopes \cite{Wetzlinger_2025} or related set-valued techniques like support functions \cite{Le_Guernic_2009}, perform a sequence of set operations at each propagation step, whose complexity depends on the state dimension and the chosen set representation.

The barrier method \cite{Levine_2013} propagates individual trajectories along the admissible set's boundary.
Although only the $(n-1)$-dimensional barrier is sampled, covering it with a prescribed resolution requires a number of trajectories that grows rapidly with the state dimension.
Moreover, trajectories can be propagated individually until their stopping points are reached \cite{Esterhuizen_2014}, resulting in trajectory-dependent propagation horizons.
In practice, a prescribed finite horizon can be used
for trajectories that do not reach a stopping point.
\end{remark}

\section{Construction of Polytopic Inner Approximation}\label{Sec:Construction_Inner_Approximation}
%
We briefly recall the notions from polytope theory required for the construction of the proposed inner approximation, see, e.g., \cite{Borrelli_Bemporad_Morari_2017} for a comprehensive overview.

A polytope $\mathcal{P} \subset \R^n$ admits two equivalent representations, the halfspace ($\mathcal{H}$-)representation and the vertex ($\mathcal{V}$-)representation \cite{Ziegler_2012}.

\begin{definition}[$\mathcal{H}$-Polytope]\label{Def:H-Polytope}
A set $\mathcal{P} \subset \R^n$ is called an $\mathcal{H}$-\emph{polytope} if it can be written as a bounded intersection of finitely many closed halfspaces.
Equivalently, there exists a $J \in \mathbb{N}$, a matrix $A_{\mathcal{P}} \in \R^{J \times n}$ and a vector $b_{\mathcal{P}} \in \R^J$ such that
\begin{align*}
\mathcal{P} := \lbrace x \in \R^n \, \vert \, A_{\mathcal{P}} x \leq b_{\mathcal{P}} \rbrace.
\end{align*}
\end{definition}

\begin{definition}[$\mathcal{V}$-Polytope]\label{Def:V-Polytope}
A set $\mathcal{P} \subset \R^n$ is called a $\mathcal{V}$-\emph{polytope} if it is the convex hull of a finite set of points $\mathcal{S} \subset \R^n$, i.e.,
\begin{align*}
\mathcal{P} := \textnormal{conv}(\mathcal{S}).
\end{align*}
\end{definition}

For each individual state constraint $i \in \{1,2,\ldots,p\}$, let $\mathcal{S}_i \subseteq \AB_i$ denote the finite set of boundary samples generated by Algorithm~\ref{alg:Structured_Barrier_Sampling}.
The corresponding $\mathcal{V}$-polytope is
\begin{align*}
    \tilde{\A}_i = \textnormal{conv}(\mathcal{S}_i).
\end{align*}
Since $\mathcal{A}_i$ is convex by Proposition~\ref{Prop:Ai_convex}, we immediately obtain $\tilde{\mathcal{A}}_i \subseteq \mathcal{A}_i$.
This gives the following result.

\begin{lemma}\label{Lem:Mi_subset_Ai}
Let $\A_i$ be the individual admissible set of system \eqref{eq:Linear_System} for some $i \in \{1,2,\ldots,p\}$.

For any finite set $\mathcal{S}_i = \{ x_1,x_2,\ldots,x_N \} \subseteq \AB_i$ of $N \in \N$ points holds
\begin{align*}
	\textnormal{conv}(\mathcal{S}_i) \subseteq \A_i.
\end{align*}
\end{lemma}

To obtain a full-dimensional approximation, the samples are required to satisfy
\begin{align*}
    \dim(\textnormal{aff}(\mathcal{S}_i)) = n.
\end{align*}

The $\mathcal{V}$-representation is subsequently converted to an
$\mathcal{H}$-representation for the intersection of the individual
approximations.
Since
\begin{align*}
    \A = \bigcap_{i=1}^{p}\A_i,
\end{align*}
the individual inner approximations yield
\begin{align*}
    \tilde{\A} = \bigcap_{i=1}^{p}\tilde{\A}_i \subseteq \A.
\end{align*}

\subsection{Conversion of Samples to Polytope}
Let $\mathcal{S}=\{x_1,\ldots,x_N\} \subseteq \AB_i$ denote the samples associated with an individual admissible set
$\mathcal{A}_i$.
The convex hull
\begin{align*}
\mathcal{P}=\operatorname{conv}(\mathcal{S})
\end{align*}
is computed using QuickHull \cite{Barber_1996}, e.g., MATLAB's \textit{convhulln}.

For each facet $\mathcal{F}_j(\mathcal{P})$, QuickHull provides the
corresponding facet vertices $\mathcal{V}_j(\mathcal{P})$.
A normal vector $a_j$ is computed for each facet from these vertices and the associated offset is given by
\begin{align*}
b_j=a_j^\top v,
\end{align*}
for any vertex $v \in \mathcal{V}_j(\mathcal{P})$.
The orientation of $a_j$ is chosen such that
\begin{align*}
a_j^\top C \leq b_j, \quad \forall \, j = 1,2,\ldots,J
\end{align*}
using the centroid $C$ of the samples in $S$.
Since $C \in \operatorname{int}(\operatorname{conv}(\mathcal S))$, this selects the interior side of $\mathcal{P}$ for each facet.
The resulting halfspace representation is therefore
\begin{align*}
\mathcal{P} = \left\{ x \in \mathbb{R}^n \; \vert \; a_j^\top x \leq b_j,\quad j=1,\ldots,J \right\}.
\end{align*}
The complete conversion from the sampled boundary points to the
$\mathcal{H}$-representation is summarized in Algorithm~\ref{alg:QuickHull_Approximation}.

\begin{algorithm}[htbp]
	\caption{Conversion of Sample Points to Polytope in Halfspace Representation}\label{alg:QuickHull_Approximation}
	\noindent \textbf{Input}: finite set of $N$ points $\mathcal{S} = \{x_1, \dots, x_N\} \subset \mathbb{R}^n$ with $N \geq n+1$ \\
	\textbf{Output:} matrix $\AP \in \R^{J \times n}$ and vector $\bP \in \R^J$ with $J \in \N$ characterizing the halfspace representation of the polytope $\mathcal{P} := \mathrm{conv}(\mathcal{S}) = \{x \in \R^n \, \vert \, \AP x \leq \bP \}$
	\begin{enumerate}
		\item Verify $\mathcal{S}$ spans an $n$-dimensional subspace, i.e., $\dim(\textnormal{aff}(\mathcal{S})) = n$. If the samples are degenerate, $\mathrm{conv}(\mathcal{S})$ cannot be represented by an $n$-dimensional polytope and the algorithm terminates.
		\item Determine the number $J$ of facets of $\mathcal{P}$ and compute the sets of vertices $\mathcal{V}_j(\mathcal{P}) \subset \mathcal{S}$ for all facets $\mathcal{F}_j(\mathcal{P})$, $j = 1,2,\ldots,J$, using a QuickHull algorithm.
		\item Compute the centroid $C$ of the set $\mathcal{S}$ as $C = \frac{1}{N} \sum_{i = 1}^N \, x_i$.
		\item \textbf{For} each facet $\mathcal{F}_j(\mathcal{P})$, $j = 1,2,\ldots,J$, of $\mathcal{P}$:
		\begin{enumerate}
    			\item Compute a normal vector $a_j$ to the facet $\mathcal{F}_j(\mathcal{P})$.
    			\item Compute offset $b_j = a_j^\top v$ for some $v \in \mathcal{V}_j(\mathcal{P})$ .
    			\item \textbf{If} $a_j^\top C > b_j$, flip normal orientation: $a_j \gets -a_j$, $b_j \gets -b_j$.
    		\end{enumerate}
		\item Return $\AP = [a_1, \ldots, a_J]^\top$, $\bP = [b_1, \ldots, b_J]^\top$.
	\end{enumerate}
\end{algorithm}

%
%

\begin{remark}
Existing results on the approximation of smooth convex bodies by polytopes constructed from randomly selected boundary samples show that the expected volume error decreases asymptotically as $N^{-2/(n-1)}$, where $N$ denotes the number of boundary samples and $n$ the dimension of the state space \cite{Gruber_1988,Schuett_2003}.
Although these assumptions are not satisfied by the proposed deterministic sampling strategy, the result indicates that the approximation accuracy is closely related to the density and distribution of the boundary samples.
Since the proposed method generates structured boundary samples that reflect the geometry of the admissible set, establishing analogous approximation guarantees for the resulting polytopic inner approximations constitutes an interesting direction for future research.
\end{remark}

\subsection{Runtime Complexity of Conversion Algorithm} 
Let $N$ denote the number of samples in $\mathcal{S}$, $n$ the state dimension and $J$ the number of facets of the resulting polytope.
We use the same assumptions on classical dense linear algebra operations from \cite[Chapter 1 ,5]{Golub_2013} and \cite{Vaidya_1987} as in Section \ref{Sec:Runtime_Sampling}.

In Step 1 of Algorithm~\ref{alg:QuickHull_Approximation}, the condition $\mathrm{dim}(\mathrm{aff}(\mathcal{S})) = n$ is verified by computing the rank of
\begin{align*}
\left[ x_2 - x_1, x_3 - x_1, \ldots, x_N - x_1 \right] \in \R^{n \times (N-1)}.
\end{align*}
Using SVD or QR decomposition requires $\mathcal{O}(n^2 N)$ operations.

Step 2, the convex hull computation, dominates the computational cost.  
Since the samples in $\mathcal{S}$ lie on the boundary of the admissible set, all points may be vertices of the convex hull. 
Consequently, the QuickHull algorithm exhibits its worst-case behavior, which scales as $\mathcal{O}(N^2)$ for $n=2,3$ \cite{Barber_1996}.  
For higher dimensions, the complexity grows with the number of facets $J$ of the polytope, which can scale as $\mathcal{O}(N^{\lfloor n/2 \rfloor})$ according to the Upper Bound Theorem.

The computation of the centroid in Step 3 requires $\mathcal{O}(n N)$ operations.

In Step 4, each of the $J$ facets requires the computation
of a normal vector, an offset, and an orientation test.
The normal computation has complexity $\mathcal{O}(n^3)$, while the remaining operations require $\mathcal{O}(n)$.
Hence, Step 4 has complexity $\mathcal{O}(Jn^3)$.


The individual complexity estimates derived above yield the following overall runtime bound for Algorithm~\ref{alg:QuickHull_Approximation}.

\begin{proposition}\label{Prop:Complexity_QuickHull}
Under the assumptions stated in Section~\ref{Sec:Runtime_Sampling}, Algorithm~\ref{alg:QuickHull_Approximation} has runtime complexity
\begin{align*}
\mathcal{O}\!\left(n^2N + N^2 + nN + Jn^3\right)
\end{align*}
for two- and three-dimensional state spaces, where $N$ denotes the number of boundary samples, $n$ the state dimension, and $J$ the number of facets of the resulting polytope.

For higher-dimensional state spaces, the worst-case complexity of the convex hull computation is bounded by
\begin{align*}
\mathcal{O}\!\left(N^{\lfloor n/2 \rfloor}\right).
\end{align*}
\end{proposition}

\begin{proof}
The result follows by summing the computational costs of the individual steps derived above.
\end{proof}

The convex hull computation therefore dominates the runtime for large $N$.
Since the worst-case number of facets grows exponentially with the state dimension, approximate convex hull algorithms may be necessary for high-dimensional systems.

\section{Examples}\label{Sec:Examples}
The proposed methods are demonstrated on a triple integrator and a
higher-dimensional mass-spring-damper chain.

All simulations were performed in MATLAB R2025b on a Linux server equipped with two 128-core AMD Turin processors (3.30 GHz) and 1.5 TB RAM.
The MATLAB process used 24 CPU cores and 40 GB RAM.
No parallel computing or GPU acceleration was used.

\subsection{Triple Integrator}
The triple integrator is used to illustrate the proposed sampling
procedure and the construction of the polytopic inner approximation.

Consider the system
\begin{align}\label{eq:Triple_Integrator}
\begin{split}
\dot{x}(t) &= \begin{pmatrix}
0 & 1 & 0 \\
0 & 0 & 1 \\
0 & 0 & 0
\end{pmatrix} x(t) + \begin{pmatrix}
0 \\ 0 \\ 1
\end{pmatrix} u(t)
\end{split}
\end{align}
with control input $u(t) \in \U = [-1,1]$ for all $t$.
Additionally, consider the affine state constraints
\begin{align}\label{eq:Triple_Integrator_Constraints}
H x + h \leq 0
\end{align}
with
\begin{align*}
H = \begin{pmatrix}
 1 &  1 &  1\\
-1 &  1 &  1\\
 1 & -1 &  1\\
 1 & -1 & -1\\
-1 &  1 & -1\\
-1 & -1 & -1
\end{pmatrix}, \qquad h = \begin{pmatrix}
-2\\
-2\\
-2\\
-2\\
-2\\
-2
\end{pmatrix}.
\end{align*}
We first apply the theory of barriers as described in Section \ref{Sec:Barrier_Theoretic_Characterization_of_the_Admissible_Set}.
For illustration, consider the first state constraint.
Condition~\eqref{eq:Condition_ultimate_tangentiality} yields
\begin{align*}
0 = \min_{u\in[-1,1]} H_1 (A\zeta_1+Bu) = \min_{u\in[-1,1]} (\zeta_{1,2}+\zeta_{1,3}+u) = \zeta_{1,2}+\zeta_{1,3}-1,
\end{align*}
and thus $\zeta_{1,3} = 1 - \zeta_{1,2}$.
Together with the boundary condition $ H_1 \zeta_1 + h_1 = 0,$ we obtain
\begin{align*}
\zeta_{1,1} + \zeta_{1,2} + \zeta_{1,3} -2 = \zeta_{1,1} + \zeta_{1,2} + (1-\zeta_{1,2}) - 2 = \zeta_{1,1} - 1 = 0.
\end{align*}
Therefore, the set of corresponding ultimate tangentiality points is given by
\begin{align*}
\ZZ_1 = \{ \zeta_1 \in \R^3 \; \vert \; \zeta_{1,1} = 1, \zeta_{1,3} = 1-\zeta_{1,2} \},
\end{align*}
where the individual ultimate tangentiality points are parameterized by $\zeta_{1,2} \in \R$.
Analogously, the sets for the remaining states are obtained as
\begin{align*}
\ZZ_2 &= \{ \zeta_2 \in \R^3 \; \vert \; \zeta_{2,2} = 2 \zeta_{2,2} - 1, \zeta_{2,3} = 1+\zeta_{2,2} \}, \\
\ZZ_3 &= \{ \zeta_3 \in \R^3 \; \vert \; \zeta_{3,1} = 2, \zeta_{3,3} = \zeta_{3,2}-1 \}, \\
\ZZ_4 &= \{ \zeta_4 \in \R^3 \; \vert \; \zeta_{4,1} = 2 \zeta_{4,2} + 2, \zeta_{4,3} = \zeta_{4,2}-1 \}, \\
\ZZ_5 &= \{ \zeta_5 \in \R^3 \; \vert \; \zeta_{5,1} = -1, \zeta_{5,3} = 1+\zeta_{5,2} \}, \\
\ZZ_6 &= \{ \zeta_6 \in \R^3 \; \vert \; \zeta_{6,1} = -2, \zeta_{6,3} = -1-\zeta_{6,2} \}.
\end{align*}
The barrier control law $\bar{u}$ is derived from condition \eqref{eq:Cond_Hamiltonian}, which reads
\begin{align*}
0 = \min_{u \in [-1,1]} \lambda(t)^\top \begin{pmatrix} x_2(t) \\ x_3(t) \\ u(t) \end{pmatrix} = \min_{u \in [-1,1]} \lambda_1(t) x_2(t) + \lambda_2(t) x_3(t) + \lambda_3(t) u(t).
\end{align*}
It follows
\begin{align*}
\bar{u}(t) = -\textnormal{sign}(\lambda_3(t)).
\end{align*}
The adjoint system \eqref{eq:Adjoint_Equation} is given by
\begin{align*}
\dot{\lambda}(t) = \begin{pmatrix}
0 & 0 &0 \\ -1 & 0 & 0 \\ 0 & -1 & 0
\end{pmatrix} \lambda(t),
\end{align*}
with the initial value $\lambda_{0i} = \nabla g_i(\zeta)$ for $\zeta_i \in \ZZ_i$, $i = 1,2,\ldots,6$.

When applying Algorithm \ref{alg:Structured_Barrier_Sampling} to sample the boundary of the admissible set of system \eqref{eq:Triple_Integrator} subject to the state constraints \eqref{eq:Triple_Integrator_Constraints}, we chose a time horizon of $20$s and a step size of $0.05$s, resulting in $K = 400$ time steps.
We used a uniform grid on $[-100,100]$ with a grid spacing of $1$, resulting in $G = 201$ grid points per time step, and restricted the area of interest to $\mathcal{X}_{\textnormal{ROI}} = [-10,10]^3 \subseteq \R^3$.
The computation time for sampling was $0.0408$s.
The resulting boundary samples for the first and second the state constraint are shown in Fig.~\ref{fig:Triple_Integrator_Samples}.
Within $\mathcal{X}_{\textnormal{ROI}}$, Algorithm~\ref{alg:Structured_Barrier_Sampling} generated $2136$ samples for the first constraint and $2650$ for the second.
Across all six state constraints, a total of $12812$ boundary samples were obtained. 

For comparison, Fig.~\ref{fig:Triple_Integrator_Samples}(c) shows samples obtained with the barrier-theoretic framework \cite{Levine_2013} using a dense uniform sampling of ultimate tangentiality points $\mathcal Z_2$.
In contrast to the proposed sampling method shown in Fig.~\ref{fig:Triple_Integrator_Samples}(b), the lower part of the admissible set boundary is almost entirely missing.
The difference is highlighted in the zoomed comparison in Fig.~\ref{fig:Triple_Integrator_zoomed}.
This effect is caused by the fact that all barrier trajectories corresponding to this part of the boundary originate from a small neighborhood of the point $\tilde{x} = (-5;-2;-1)^\top$ in the ultimate tangentiality set $\mathcal{Z}_2$.

This phenomenon can be explained by analyzing the evolution of the vector field of system \eqref{eq:Triple_Integrator} under the barrier control.
For points in the ultimate tangentiality set $\mathcal{Z}_2$, the first time derivative of the constraint vanishes by definition.
At the point $\tilde{x}$, however, also the second time derivative vanishes, corresponding to a higher-order tangentiality condition.
Consequently, the temporal extremum of $g(x(t))$ in \cite[Proposition~5.1(iii)]{Levine_2013} changes from a supremum to an infimum, separating barrier-generating trajectories from trajectories that leave the admissible set when integrated backward in time.
As a result, barrier trajectories originating from nearby ultimate tangentiality points evolve almost identically in a neighborhood of $\tilde{x}$, leading to a local loss of sampling resolution.
While this phenomenon may occur for any state constraint, its practical impact depends on the subsequent evolution of the corresponding barrier trajectories.

The constructed polytopic inner approximation of the admissible set $\mathcal{A}_2$, obtained by using Algorithm~\ref{alg:QuickHull_Approximation}, has $492$ facets and is shown in Fig.~\ref{fig:Triple_Integrator_Reconstruction}(a).
Figure~\ref{fig:Triple_Integrator_full_Reconstruction}(b) illustrates the proposed decomposition approach by showing the intersection of the polytopic inner approximations of the admissible sets $\mathcal{A}_2$ and $\mathcal{A}_4$. 
The combined representation $\mathcal{A}=\cap_{i=1}^6 \mathcal A_i$, obtained by intersecting the polytopic inner approximations associated with all six state constraints, contains $3584$ facets and was computed in $0.0451$s. Redundant inequalities are removed to obtain a minimal $\mathcal{H}$-representation using the algorithm of \cite{Klintberg_2018}. Computing the minimal representation required $0.3011$s and reduced the number of inequalities from $3584$ to $292$.
The resulting minimal representation of $\mathcal{A}$ is shown in Fig.~\ref{fig:Triple_Integrator_full_Reconstruction}.
In the figures, the black lines indicate the edges of the reconstructed polytope, while the white lines indicate the intersections of the state-constraint hyperplanes.

\begin{figure}[htbp]
    \centering
    \begin{subfigure}{0.79\textwidth}
        \centering
        \includegraphics[width=\linewidth]{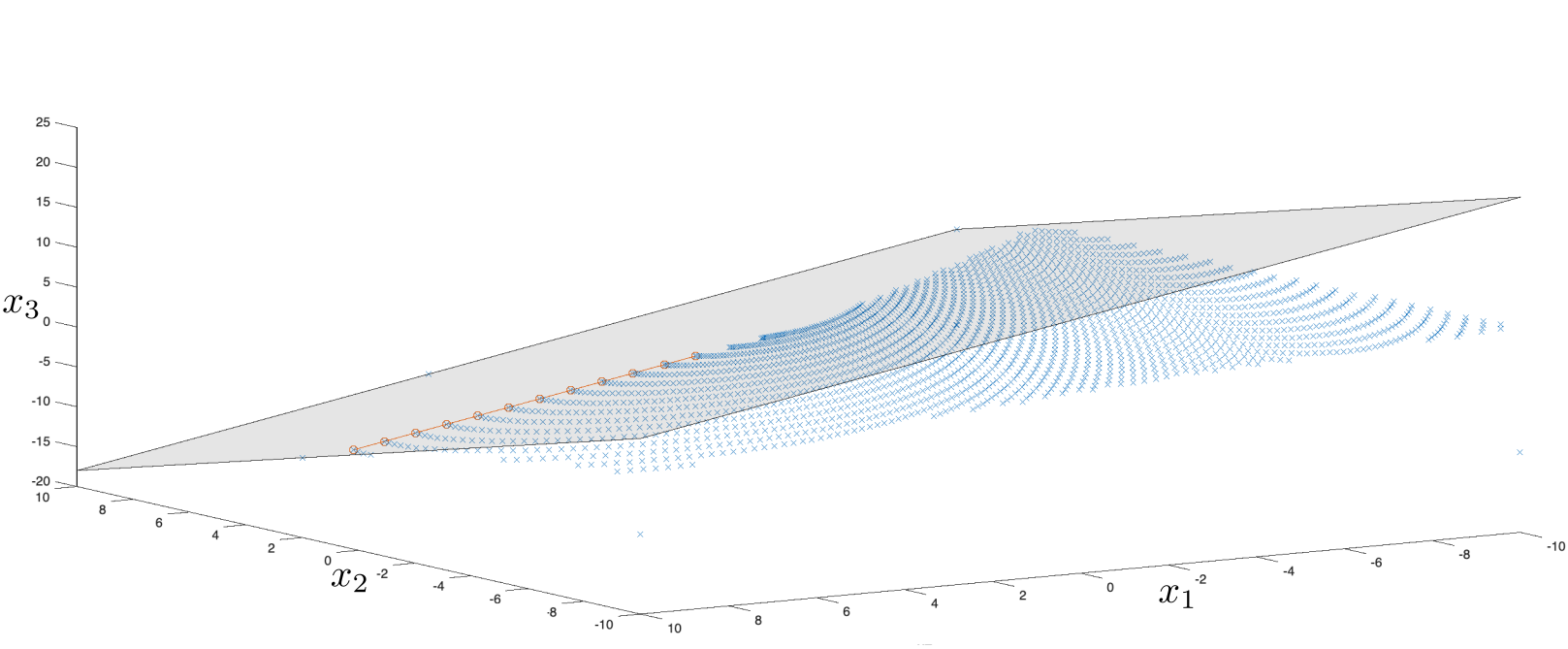}
        \caption{Boundary samples of admissible set for state constraint $H_1 x + h_1 \leq 0$ generated by Algorithm~\ref{alg:Structured_Barrier_Sampling}.}
    \end{subfigure}
    \\
    \begin{subfigure}{0.79\textwidth}
        \centering
        \includegraphics[width=\linewidth]{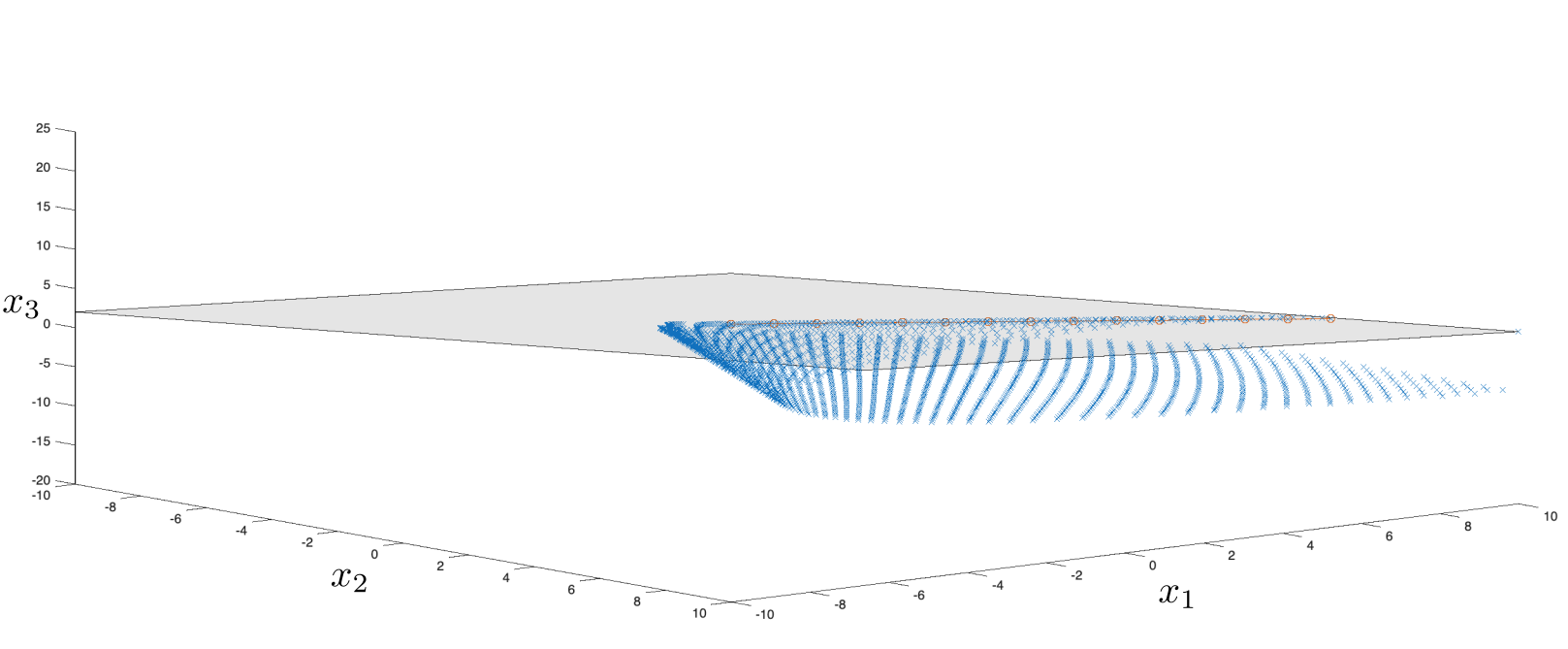}
        \caption{Boundary samples of admissible set for state constraint $H_2 x + h_2 \leq 0$ generated by Algorithm~\ref{alg:Structured_Barrier_Sampling}.}
    \end{subfigure}
    \hfill
    \begin{subfigure}{0.79\textwidth}
        \centering
        \includegraphics[width=\linewidth]{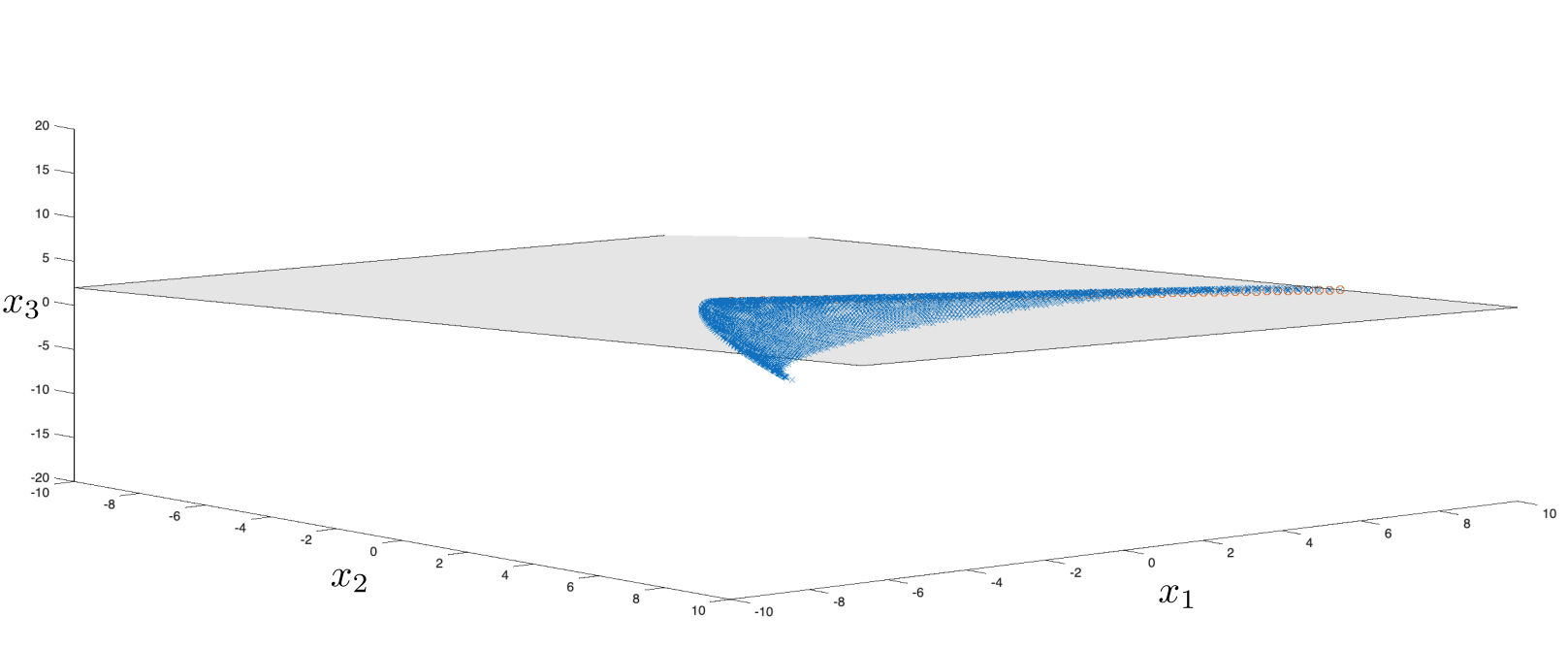}
        \caption{Boundary samples of admissible set for state constraint $H_2 x + h_2 \leq 0$ generated with the barrier-theoretic framework \cite{Levine_2013}.}
    \end{subfigure}
\caption{Boundary samples of the admissible set for the triple integrator \eqref{eq:Triple_Integrator}. Boundary samples are shown in blue, ultimate tangentiality points in red, and state constraints in gray.}
    \label{fig:Triple_Integrator_Samples}
\end{figure}

\begin{figure}[htbp]
    \centering
    \begin{subfigure}{0.79\textwidth}
        \centering
        \includegraphics[width=\linewidth]{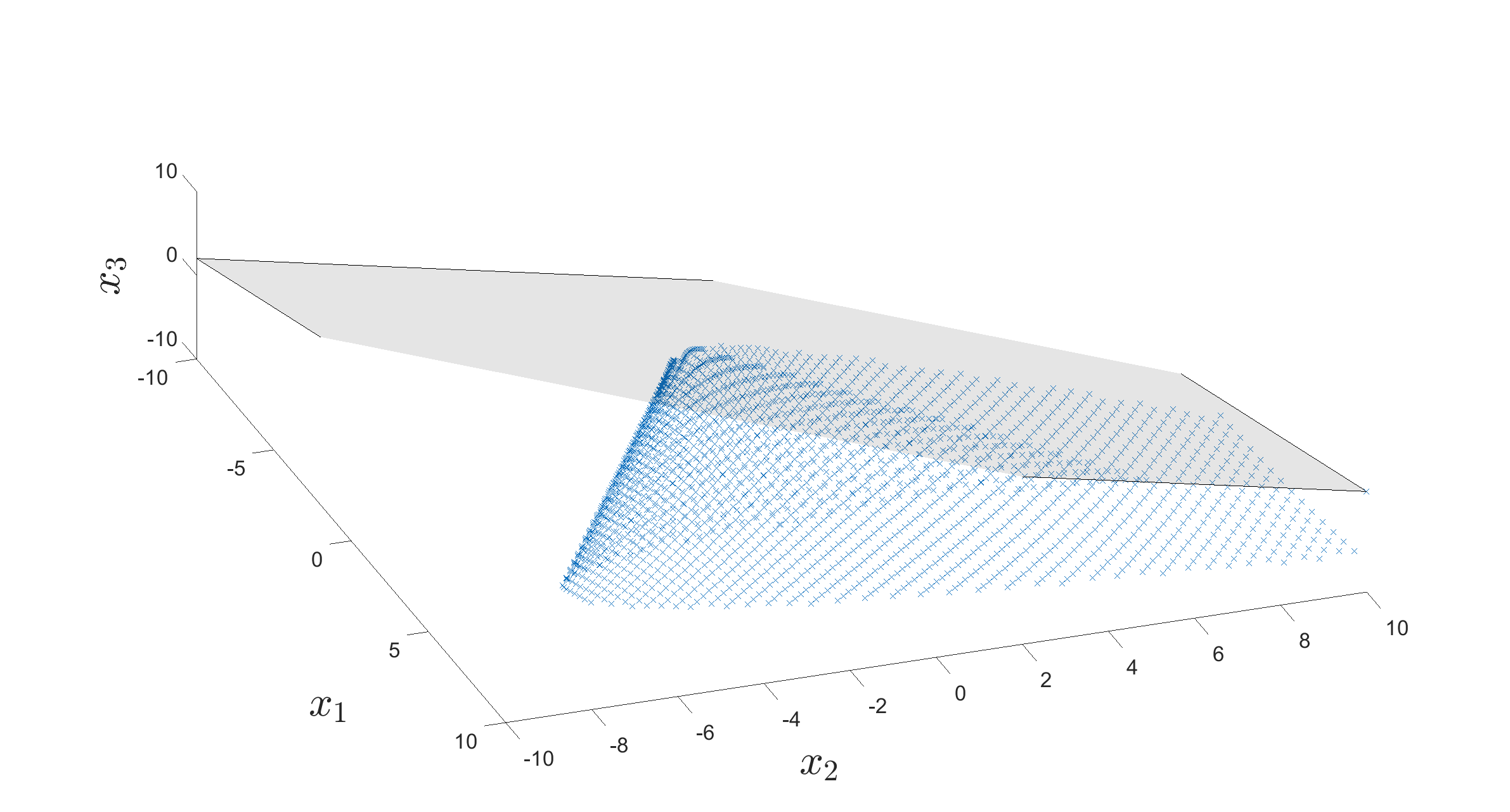}
        \caption{Zoomed view of the boundary samples for state constraint $H_2x+h_2\leq0$ generated by Algorithm~\ref{alg:Structured_Barrier_Sampling}.}
    \end{subfigure}
    \\
    \begin{subfigure}{0.79\textwidth}
        \centering
        \includegraphics[width=\linewidth]{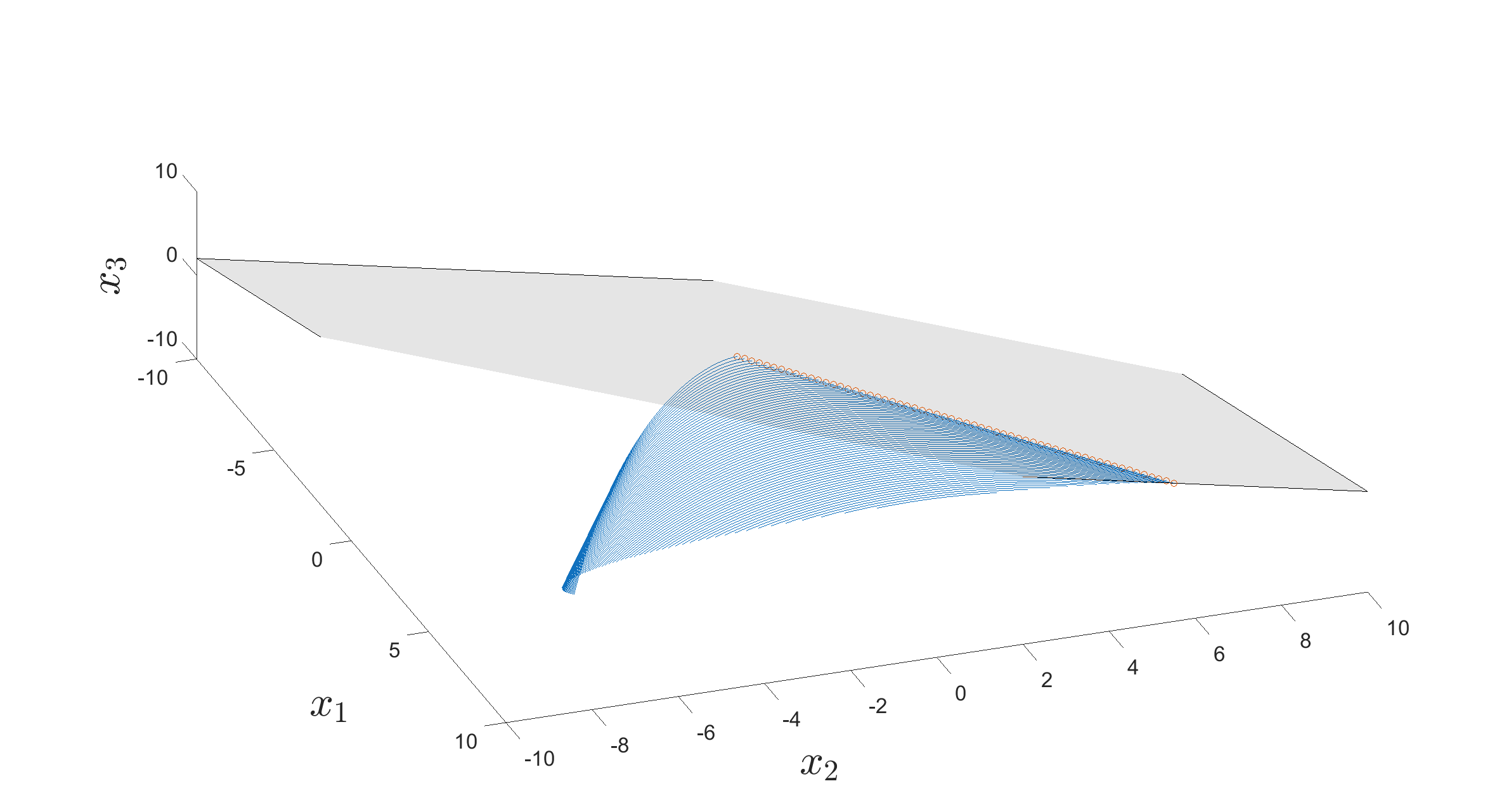}
        \caption{Zoomed view of the boundary samples for state constraint $H_2x+h_2\leq0$ generated with the barrier-theoretic framework \cite{Levine_2013}.}
    \end{subfigure}
\caption{Zoomed comparison of the boundary samples for the triple integrator \eqref{eq:Triple_Integrator}. Boundary samples are shown in blue, ultimate tangentiality points in red, and state constraints in gray.}
    \label{fig:Triple_Integrator_zoomed}
\end{figure}

\begin{figure}[htbp]
    \centering
    \begin{subfigure}{0.95\textwidth}
        \centering
        \includegraphics[width=\linewidth]{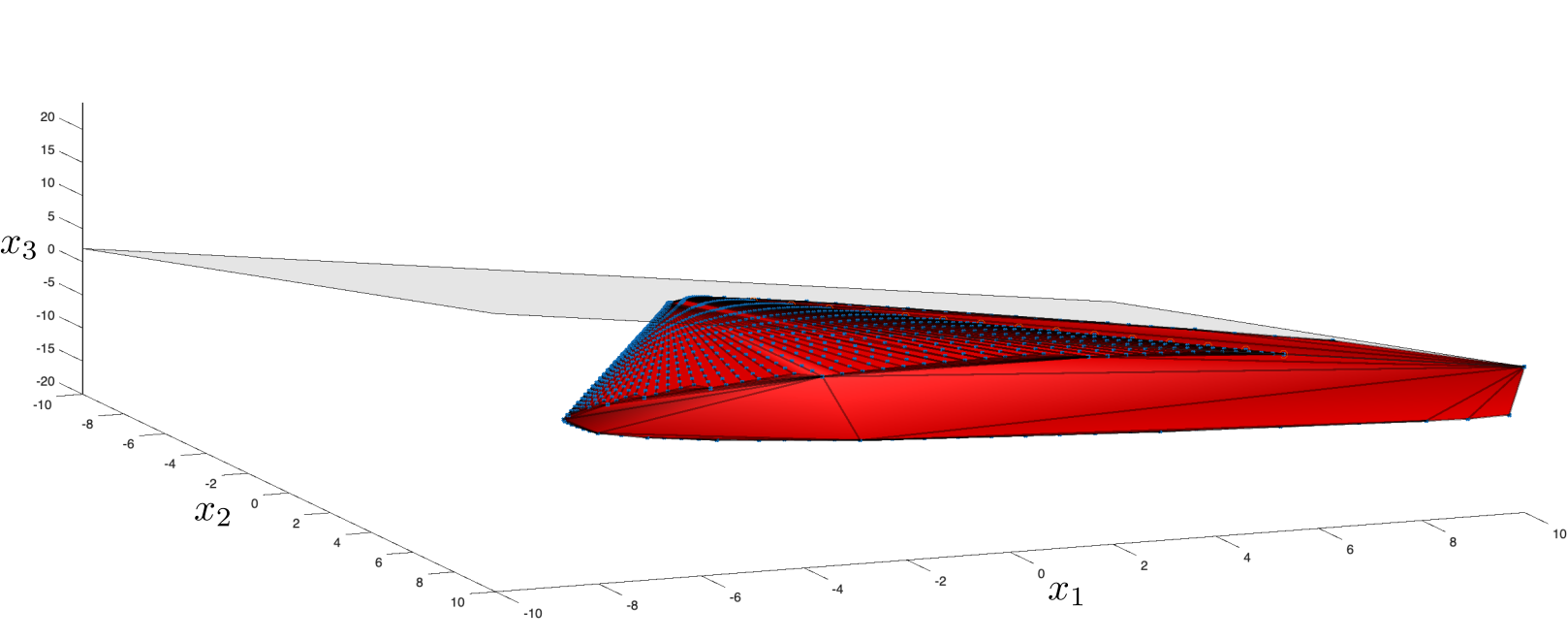}
        \caption{Polytopic inner approximation of admissible set for constraint $H_2 x + h_2 \leq 0$ in red. Boundary samples are shown in blue, state constraint in gray.}
    \end{subfigure}
    \hfill
    \begin{subfigure}{0.95\textwidth}
        \centering
        \includegraphics[width=\linewidth]{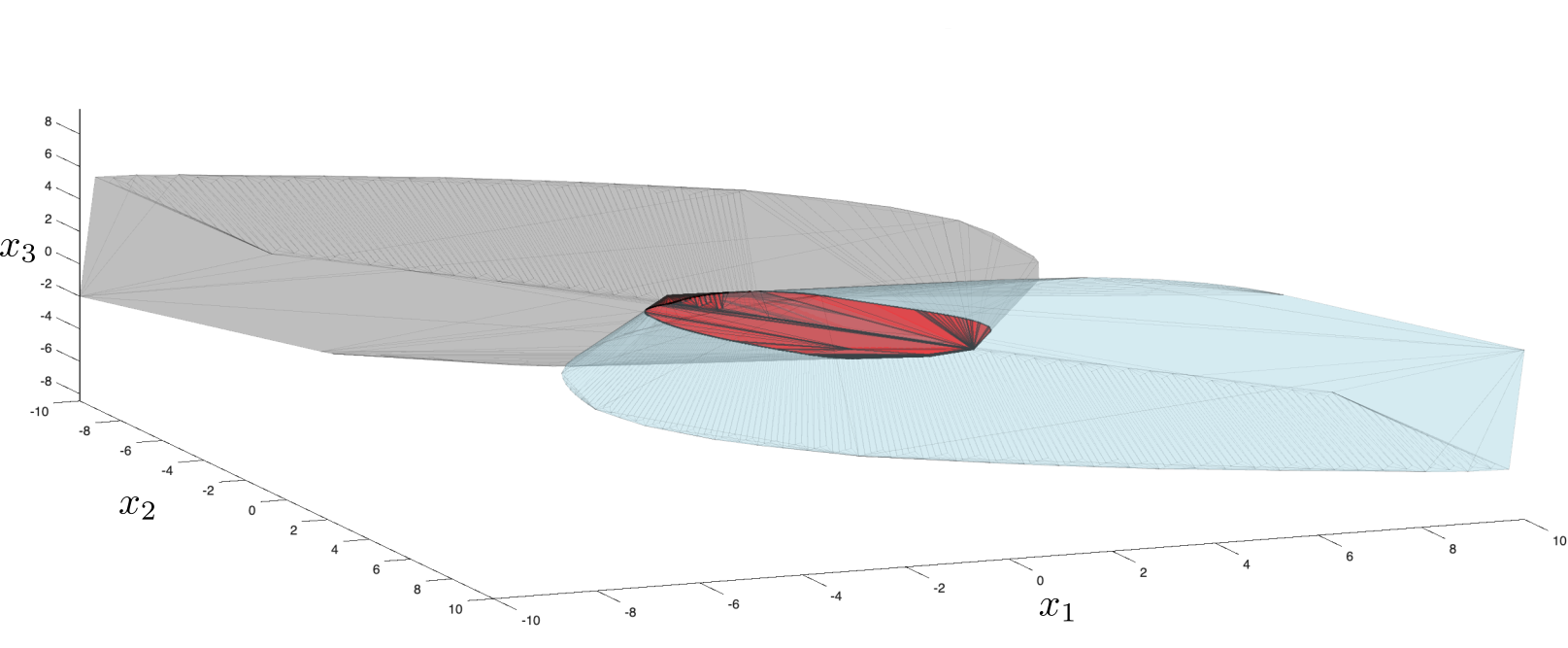}
        \caption{Polytopic inner approximations of the admissible sets associated with the state constraints $H_2x+h_2\leq0$ (light blue) and $H_4x+h_4\leq0$ (gray). Their intersection is shown in red.}
    \end{subfigure}
\caption{Polytopic inner approximations of the admissible set for the triple integrator \eqref{eq:Triple_Integrator}.}
    \label{fig:Triple_Integrator_Reconstruction}
\end{figure}
\begin{figure}[htbp]
    \centering
\includegraphics[width=\linewidth]{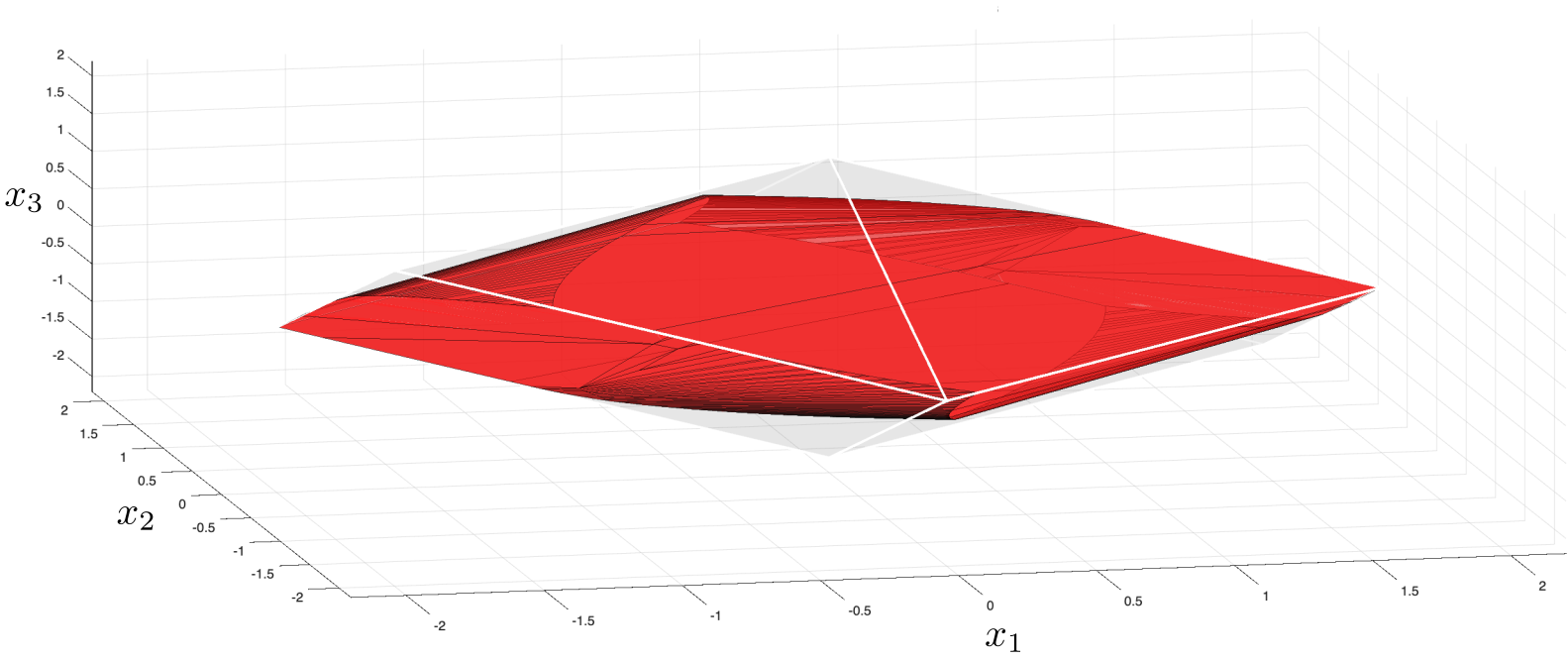}
        \caption{Polytopic inner approximation of admissible set subject to state constraints $Hx + h \leq 0$. Polytopic inner approximation in red, and state constraints in gray.}\label{fig:Triple_Integrator_full_Reconstruction}
\end{figure}

\subsection{Mass-Spring-Damper Chain}
The mass-spring-damper chain serves as a higher-dimensional benchmark for evaluating the runtime and scalability of the proposed algorithms.

Let $q_i$ denote the displacement of the $i$-th mass from its equilibrium position and $v_i = \dot q_i$ its velocity, $i = 1,\ldots,n$.
The dynamics of the mass--spring chain are described by the second-order system
\begin{align}\label{eq:MSD_Chain}
M_q \ddot{q}(t) + D_q \dot{q}(t) + K_q q(t) = B_q u(t) ,
\end{align}
where $q = (q_1,q_2,\ldots,q_n)^\top$ is the vector of displacements, $M_q \in \R^{n \times n}$ is the diagonal mass matrix, $D_q \in \R^{n \times n}$ the damping matrix, and $K_q \in \R^{n \times n}$ the stiffness matrix describing the coupling between adjacent masses.
The input $u(t) \in \R^m$ models an external force, whose effect on the system is captured by the input distribution matrix $B_q \in \R^{n \times m}$.

Assuming identical masses, spring constants, and damping coefficients, the matrices take the form
\begin{align*}
M_q &= m I_n, \; K_q = k \begin{pmatrix}
2 & -1 &  &  \\
-1 & 2 & \ddots & \\
 & \ddots & \ddots & -1 \\
 &  & -1 & 2
\end{pmatrix}, \; D_q = d \begin{pmatrix}
2 & -1 &  &  \\
-1 & 2 & \ddots & \\
 & \ddots & \ddots & -1 \\
 &  & -1 & 2
\end{pmatrix},
\end{align*}
where the omitted entries are zero. Introducing the state vector
\begin{align*}
x = \begin{pmatrix}
q \\
v
\end{pmatrix} \in \R^{2n} ,
\end{align*}
where
\begin{align*}
    q = \left(q_1,q_2,\ldots,q_n \right)^\top, \; v = \left(v_1,v_2,\ldots,v_n \right)^\top,
\end{align*}
the system can be written in first-order state-space form $\dot x = A x + B u$ with
\begin{align*}
A = \begin{pmatrix}
0 & I_n \\
-M_q^{-1}K_q & -M_q^{-1}D_q
\end{pmatrix},
\qquad
B = \begin{pmatrix}
0_n \\
M_q^{-1} B_q
\end{pmatrix}.
\end{align*}

Consider affine state constraints of the form $|q_1| \leq q_{\textnormal{max}}$, which can be written as
\begin{align*}
x_1 - q_{\max} &\leq 0 \\
-x_1 - q_{\max} &\leq 0
\end{align*}
and limit the displacements of the first mass.

In our simulations, we chose the parameters
\begin{align*}
    m = 1, \; k = 20, \; d = 0.05, \; q_{\textnormal{max}} = 1
\end{align*}
and use the input matrix $B = \left( 0, \ldots, 0, 1 \right)^\top$ with the input constraints $u(t) \in [-1,1]$.

We performed simulations for $n = 2,3,4,5$, corresponding to systems with $4$, $6$, $8$ and $10$ states, respectively.
The polytopic inner approximation was restricted to the region of interest $\mathcal{X}_{\textnormal{ROI}} = [-10,10]^{2n}$.
A time horizon of $4$s with sampling time of $0.2$s was used, resulting in $K = 20$ time steps.
The grid spacing was adjusted for each dimension to balance the number of generated boundary samples with the computational tractability of the subsequent polytope construction. The simulation results are summarized in Table~\ref{Tab:MSD_Chain}.

\begin{table}[h]
\caption{Simulation Results Mass-Spring-Damper Chain }\label{Tab:MSD_Chain}%
\begin{tabular}{@{}l|lllll@{}}
\toprule
Masses & $2$ & $3$ & $4$ & $4$ & $5$ \\
\midrule
Gird Spacing & $0.2$ & $2$ & $10$ & $4$ & $4$ \\
Samples & $138,902$ & $104,616$ & $1,656$ & $163,278$ & $4,369,962$ \\
Polytope Facets & $1,920$ & $230,842$ & $403,195$ & $-^{\ast}$ & $-^{\ast}$ \\
Reduced Facets & $1,420$ & $179,252$ & $296,069$ & $-^{\ast}$ & $-^{\ast}$ \\
Time for Sampling & $0.0577$s & $0.1351$s & $0.0321$s & $0.2074$s & $64.5713$s \\
QuickHull Time & $0.3323$s & $128.2692$s & $2,083.2$s & $-^{\ast}$ & $-^{\ast}$ \\
$\mathcal{H}$-reduction Time & $0.3042$s & $2,853.6$s & $12,906.0$s & $-^{\ast}$ & $-^{\ast}$ \\
Total Time & $0.6942$s & $2,982.0$s & $14,989.2$s & $-^{\ast}$ & $-^{\ast}$ \\
\bottomrule
\end{tabular}
\footnotetext{$\ast$: Computation was terminated after 8 hours without completion of the convex hull construction.}
\end{table}

The comparatively large reduction in the number of facets for the lower-dimensional examples is mainly due to the finer sampling grid.
Since the grid spacing was changed to maintain a comparable number of boundary samples across all dimensions, many neighboring samples generated almost identical supporting hyperplanes.
As a result, a large fraction of the inequalities produced by the convex hull construction were identified as redundant.

For the $8$-dimensional system, an additional pre-processing step was required before applying the redundancy-removal algorithm.
Specifically, duplicate facets were removed by discarding inequalities whose coefficient vectors and offsets were equal to an already existing facet up to the eighth decimal place.
Without this pre-processing, the redundancy-removal algorithm of \cite{Klintberg_2018} failed because of numerically indistinguishable inequalities.

It is worth noting that the increase in the total computation time is almost entirely caused by the convex hull construction and the redundancy-removal procedure.
In contrast, the proposed structured barrier sampling algorithm requires less than $0.5$s even for the $8$-dimensional example, demonstrating that the sampling procedure itself scales well with the problem dimension. Indicating that for higher-dimensional examples an approximation of the convex hull might be required for computational feasibility.

\begin{figure}[htbp]
    \centering
    \includegraphics[width=\linewidth]{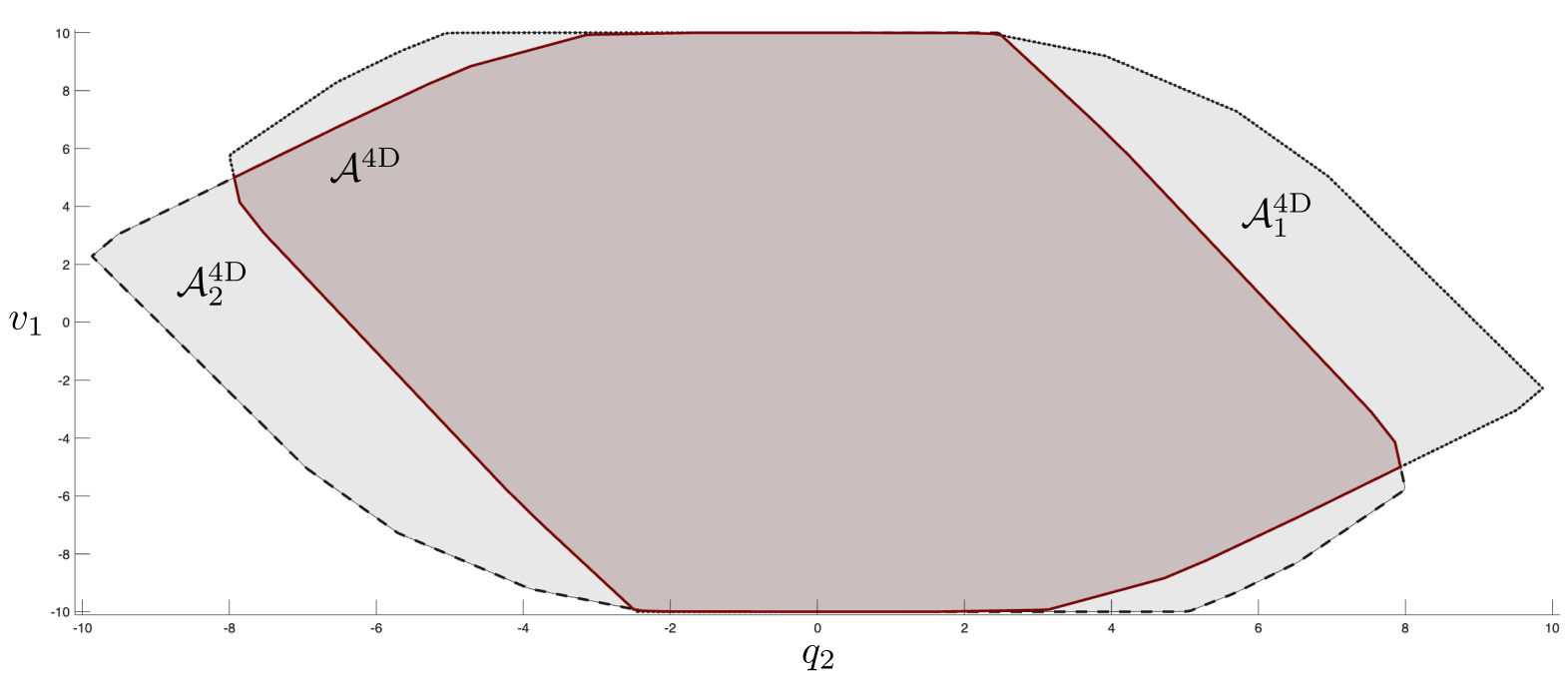}
    \caption{Slices through the 4D polytopic inner approximation of the admissible set for the mass-spring-damper chain~\eqref{eq:MSD_Chain} with two masses. The slices are obtained by setting all states to zero except for $q_2$ and $v_1$. The gray regions show the admissible sets $\mathcal{A}^{\mathrm{4D}}_1$ and $\mathcal{A}^{\mathrm{4D}}_2$ associated with the state constraints $x_1-q_{\max}\leq 0$ and $-x_1-q_{\max}\leq 0$, respectively. Their intersection, $\mathcal{A}^{\mathrm{4D}}=\mathcal{A}^{\mathrm{4D}}_1\cap\mathcal{A}^{\mathrm{4D}}_2$, is shown in red.}
    \label{fig:4D_MSD_chain}
\end{figure}
\begin{figure}[htbp]
    \centering
    \begin{subfigure}{0.95\textwidth}
    \centering
    \includegraphics[width=\linewidth]{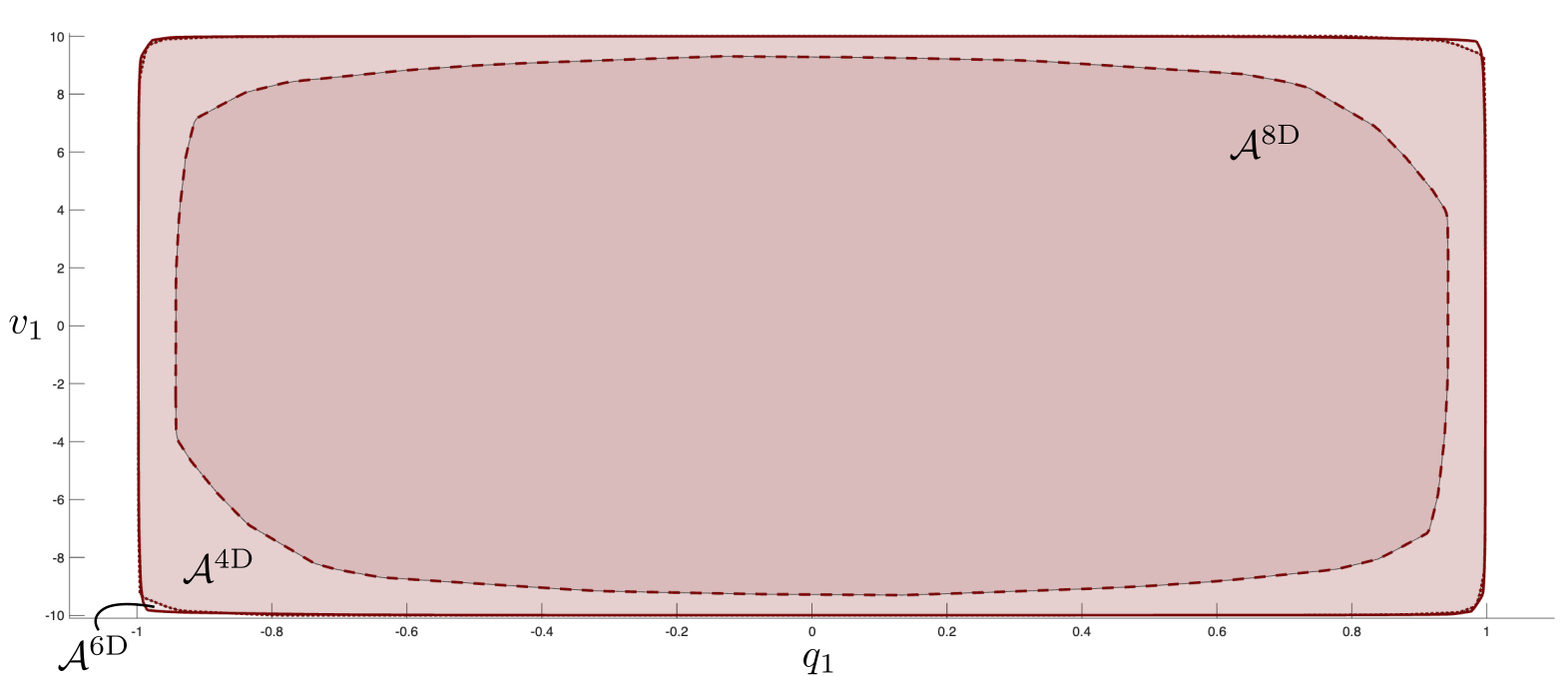}
    \caption{Slices were obtained by setting every state to zero except for $q_1$ and $v_1$.}
    \end{subfigure}
    \hfill
    \begin{subfigure}{0.95\textwidth}
    \centering
    \includegraphics[width=\linewidth]{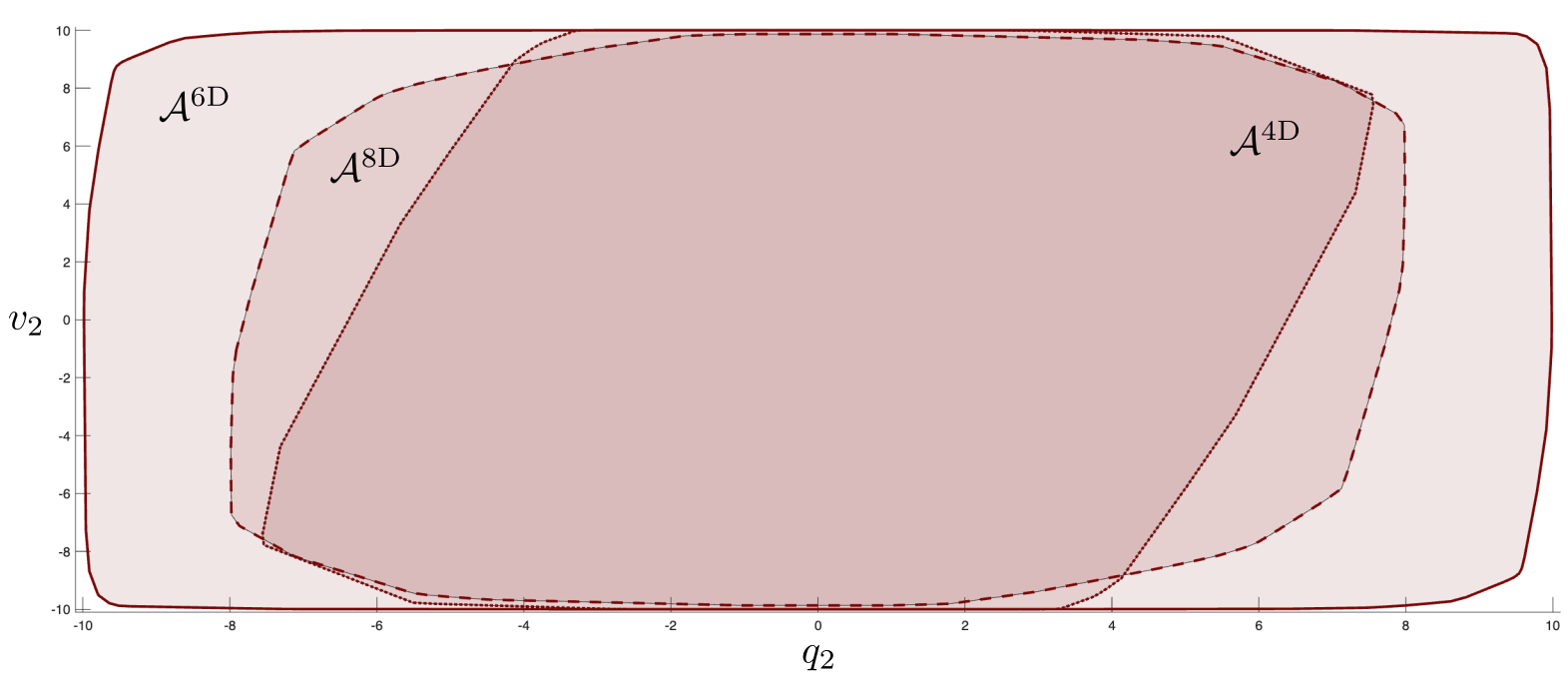}
    \caption{Slices were obtained by setting every state to zero except for $q_2$ and $v_2$.}
    \end{subfigure}
    \caption{Slices through the polytopic inner approximation of the admissible set for the mass-spring-damper chain \eqref{eq:MSD_Chain}. $\mathcal{A}^{\mathrm{4D}}, \mathcal{A}^{\mathrm{6D}}, \mathcal{A}^{\mathrm{8D}}$ correspond to the slices of the admissible set for two, three, and four masses, respectively.
    }
    \label{fig:MSD_chain}
\end{figure}

In addition to the results reported in Table~\ref{Tab:MSD_Chain}, Figs.~\ref{fig:4D_MSD_chain} and \ref{fig:MSD_chain} illustrate the resulting polytopic inner approximations and the effect of the proposed decomposition with respect to the individual state constraints.
The figures show two-dimensional sections obtained by varying two state variables while setting the remaining variables to zero, and therefore do not capture the full state-space interactions.

Fig.~\ref{fig:4D_MSD_chain} shows the proposed decomposition of the admissible set with respect to the individual state constraints for the 4D mass-spring-damper chain~\eqref{eq:MSD_Chain}.
The gray regions correspond to the admissible sets $\mathcal{A}^{\mathrm{4D}}_1$ and $\mathcal{A}^{\mathrm{4D}}_2$ associated with the constraints $x_1-q_{\max} \leq 0$ and $-x_1-q_{\max}\leq 0$ respectively, while the red region represents the final admissible set $\mathcal{A}^{\mathrm{4D}}=\mathcal{A}^{\mathrm{4D}}_1\cap\mathcal{A}^{\mathrm{4D}}_2$. 
This decomposition avoids the explicit treatment of simultaneous active constraints by first constructing the admissible sets associated with the individual state constraints and subsequently intersecting their inner polytopic approximations.

Fig.~\ref{fig:MSD_chain}(a) shows slices through the admissible sets obtained by varying only the states $q_1$ and $v_1$, while all remaining states are fixed to zero.
As the number of masses is increased to $4$, the admissible region decreases substantially, which is reflected by the increasing geometric complexity of the polytopic inner approximations resulting in a growing number of facets reported in Table~\ref{Tab:MSD_Chain}.

The slices in Fig.~\ref{fig:MSD_chain}(b), obtained by varying $q_2$ and $v_2$, illustrate that the evolution of the admissible set with increasing system dimension.
Although the admissible set for the six-dimensional system almost covers the entire displayed region, the four- and eight-dimensional admissible sets have noticeably different geometries.
In particular, neither of the latter is contained in the other, demonstrating that the influence of additional masses on the admissible set is not simply characterized by a monotonic expansion or contraction.

\section{Conclusion}\label{Sec:Conclusion}
This paper proposed a computationally efficient construction of inner polytopic approximations of admissible sets for linear control systems subject to affine state constraints.
Building upon the barrier-theoretic characterization of admissible sets in \cite{Levine_2013}, an exact sampling approach was developed to generate boundary samples directly.
The proposed method propagates sets of barrier trajectories.

The proposed framework further decomposes the admissible set computation into individual state constraints, allowing the corresponding inner polytopic approximations to be constructed separately and subsequently combined by intersection.
This avoids the explicit treatment of simultaneous barrier intersections while preserving the barrier-theoretic characterization of the admissible set.
The boundary samples are converted into half-space representations using the QuickHull algorithm, and runtime complexity analyses were derived for both the structured sampling and polytope reconstruction algorithms.
While the sampling procedure exhibits polynomial complexity in the state dimension and the number of time steps, the computational cost of the polytope reconstruction is dominated by the convex hull computation and therefore depends strongly on both the number of samples and the state-space dimension.

The proposed method was demonstrated on two examples. 
For a triple integrator subject to six affine state constraints, the structured sampling procedure and the construction of the corresponding inner polytopic approximation were illustrated in detail. 
A higher-dimensional mass-spring-damper chain was then used to demonstrate the scalability of the proposed algorithms.

Possible future applications include the extension of the proposed framework to nonlinear model predictive control. 
In particular, the computational efficiency of the structured sampling algorithm may enable the online computation of admissible set approximations based on local linearizations of nonlinear systems. 
Important questions in this context include quantifying the approximation quality, identifying conditions that guarantee the preservation of the inner approximation property, and determining the time horizons over which admissible sets computed from local linearizations remain valid for the underlying nonlinear dynamics.

\section*{Funding}
This work was supported by the Deutsche Forschungsgemeinschaft (DFG, German Research Foundation) - Project Number 531896505.

\section*{Data Availability}
No datasets were generated or analyzed during the current study.

%

\bibliographystyle{elsarticle-num} 
\bibliography{bibliography.bib}

\end{document}